\documentclass[leqno]{amsart}
\usepackage{amscd,amssymb,amsmath}

\usepackage{color}
\usepackage{float}
\usepackage{graphicx}
\usepackage{vntex}
\usepackage[brazil]{babel}    
\usepackage{tikz}

\usepackage{amsmath,amsfonts,amsthm,amssymb}
\usepackage{mathabx}

\usepackage[all]{xy}     
\numberwithin{equation}{section}

\theoremstyle{definition}
\newtheorem{theorem}[equation]{Teorema}
\newtheorem{lemma}[equation]{Lema}
\newtheorem{proposition}[equation]{Proposi\c c\~ao}
\newtheorem{remark}[equation]{Observa\c c\~ao}
\newtheorem{corollary}[equation]{Corol\'ario}
\newtheorem{definition}[equation]{Defini\c c\~ao}

\numberwithin{equation}{section}
\theoremstyle{definition}

\newtheorem{example}[equation]{Exemplo}

\newcommand{\Z}{\mathbb{Z}}
\newcommand{\R}{\mathbb{R}}
\newcommand{\Sp}{\mathbf{S}}
\newcommand{\T}{\mathbf{T}}
\newcommand{\PP}{\mathbb{P}}
\newcommand{\Su}{\mathcal{S}}

\begin{document}


\title[Teorema de Poincar\'e-Hopf] {Teorema de Poincar\'e-Hopf}

\author{Jean-Paul \textsc{Brasselet} e Nguy\~{\^e}n Th\d{i} B\'ich Th\h{u}y}
\address{CNRS I2M, Aix-Marseille Universit\'e, Marseille, France.}
\email{jean-paul.brasselet@univ-amu.fr}

\address{UNESP, Universidade Estadual Paulista, ``J\'ulio de Mesquita Filho'', S\~ao Jos\'e do Rio Preto, Brasil}
\email{bich.thuy@unesp.br}

\thanks{Este artigo \'e uma vers\~ao extendida dum curso providenciado duranta a semana XXVIII SEMAT Semana de Matem\'atica, 17 a 21 de Outobro de 2016. O primerio autor tinha ajuda duma bolsa UNESP-FAPESP no. 2015/06697-9}
\bigskip




\begin{abstract}

O Teorema de Poincar\'e-Hopf \'e um dos resultados da mat\'ematica mais usados n\~ao somente 
em mat\'ematica  mas tamb\'em em outras \'areas da ci\^encia. 
H\'a aplica\c c\~oes do Teorema de Poincar\'e-Hopf
em f\'isica, qu\'imica, biologia e mesmo em economia, psicologia, etc.

O Teorema de Poincar\'e-Hopf liga um invariante da combinat\'oria : a caracter\'\i stica de Euler-Poincar\'e 
com objetos da geometria diferencial, a saber \'indices de campos de vetores. 
Os resultados que ligam assim duas \'areas bem diferentes da matem\'atica podem ser considerados
como os mais bonitos, \'uteis e fecundos. 
\end{abstract}

\maketitle

\section{Introdu\c c\~ao}

O Teorema de Poincar\'e-Hopf \'e bem conhecido, mas sua demonstra\c c\~ao nem tanto. 
O objetivo desta apresenta\c c\~ao \'e introduzir a prova, ou ao menos, uma ideia da prova 
do Teorema de Poincar\'e-Hopf no caso de superf\'\i cies, assim como algumas consequ\^encias.

Depois de lembrar o cl\'assico Teorema  de Poincar\'e-Hopf,  
introduzimos as no\c c\~oes involvidas no Teorema: a caracter\'\i stica de Euler-Poincar\'e 
e o \'indice de um campo de vetores em uma singularidade isolada. 

Nas se\c c\~oes seguintes, damos defini\c c\~oes equivalentes, no caso diferenci\'avel,
o que permite de providenciar a prova do Teorema de Poincar\'e-Hopf. 
Em primeiro lugar, lembramos a no\c c\~ao de fibrado tangente a uma  
curva $\mathcal C$ suave e fechada em $\R^2$ (\S \ref{tangcirc}),  
assim como as no\c c\~oes de 
valores e de pontos regulares e cr\'iticos e de grau de uma aplica\c c\~ao suave $f$ 
de $\mathcal{C}$ em $\Sp^1$,  onde $\Sp^1$ \'e a circunfer\^encia unit\'aria.

Estas no\c c\~oes generalizam-se no caso de uma superf\'icie suave $\mathcal{S}$, 
em particular a no\c c\~ao de 
fibrado tangente a uma superf\'icie suave (\S \ref{2.2}). Definimos as no\c c\~oes de 
valores e pontos regulares e cr\'iticos e de grau de uma aplica\c c\~ao suave $f$ 
de $\mathcal{S}$ em $\Sp^2$ (\S \ref{grau2}).
Isso nos permite de estudar as singularidades de campos de vetores tangentes 
a uma superf\'icie suave. Providenciamos v\'arios exemplos (assim como desenhos
correspondentes). O \' indice de um campo $v$ de vetores tangentes a uma superf\'icie, em um ponto
singular isolado, definise-se na se\c c\~ao \S \ref{indice1} 
como o grau de uma aplica\c c\~ao suave $\gamma_v: \Sp^1 \to \Sp^1$ definido na se\c c\~ao  \S \ref{grau}.

Passando na dimens\~ao superior, em $\R^3$, o \' indice de um campo $v$ de vetores em $\R^3$
em um ponto singular isolado definise-se como grau de uma aplica\c c\~ao suave $\gamma_v : \Sp^2 \to \Sp^2$,
isto \'e o grau definido na se\c c\~ao precedente (\S \ref{grau2}). Mostramos neste ponto 
um resultado t\'ecnico que ser\'a usado na prova do Teorema de Poincar\'e-Hopf (Lemma \ref{extension}). 

Na se\c c\~ao seguinte enunciamos o Teorema de Poincar\'e-Hopf e, antes de providenciar a prova, 
damos v\'arios exemplos nos casos orientado e n\~ao orientado. Em fim, damos a prova mesma primeiramente 
no caso orientado, em duas etapas. A primeira etapa \'e bem geometrica. Descrevemos a constru\c c\~ao de um campo
de vetores particulares, chamado o campo de Hopf, satisfazendo o Teorema de Poincar\'e-Hopf. A segunda etapa 
\'e um pouco mais t\'ecnica, mas j\'a providenciamos todos ingredientes para mostrar que, em uma superf\'icie 
compacta orientada, a soma dos \' indices de um campo de vetores tangentes \`a 
superf\'icie n\~ ao depende do campo escolhido. Isso termine a prova no caso orientado. 
Concluimos dando a prova do Teorema de Poincar\'e-Hopf no caso n\~ao orientado. 

\section{Enunciado do Teorema de Poincar\'e-Hopf}

O Teorema de Poincar\'e-Hopf foi provado por Poincar\'e no ano 1885 
no caso de superf\'icies, depois  por Hopf, no ano 1927, para o caso de variedades suaves 
de dimens\~ao maior. 

\begin{theorem}
{\it Seja $v$ um campo de vetores com singularidades isoladas, tangente a uma superf\'icie $\Su$ compacta, suave, sem borda e n\~ao necessariamente orientada. Temos 
$$\sum_{a_i} I(v, a_i) = \chi(\Su),$$
onde os pontos $a_i$ s\~ao os pontos singulares de $v$. }
\end{theorem}

Nas duas primeiras se\c c\~oes, lembraremos a defini\c c\~ao dos ingredientes 
do Teorema que s\~ao a caracter\'\i stica de Euler-Poincar\'e e  
 o \'indice de um campo de vetores em uma singularidade isolada.
Nas se\c c\~oes seguintes daremos defini\c c\~oes equivalentes, no caso diferenci\'avel,
o que permite de providenciar a prova do Teorema de Poincar\'e-Hopf. 

\section{Caracter\'\i stica de Euler-Poincar\'e}

Historicamente,  a ``caracter\'\i stica de Euler-Poincar\'e'' de uma superf\'icie 
foi definida usando triangula\c c\~oes. 
Exemplos de triangula\c c\~oes da esfera $\Sp^2$ s\~ao dados na Figura \ref{triangsphere}.

 Uma triangula\c c\~ao de uma superf\'icie ${\mathcal S}$ \'e o dado de um poliedro (triangulado) $K\subset \R^n$ e de um homeomorfismo $h:\vert K \vert \to \Su$, onde $\vert K \vert$ \'e o 
subespa\c co topol\'ogico compacto de $\R^n$, 
que \'e a uni\~ ao dos simplexos do poliedro $K$.

Seja $K$ uma triangula\c c\~ao de uma superf\'icie compacta $\Su$, 
com  n\'umeros $n_0$ de v\'ertices, $n_1$ de segmentos e $n_2$ de tri\^angulos, 
a caracter\'\i stica de Euler-Poincar\'e de $K$ \'e definida por
\begin{equation}\label{Euler}
\chi(K)=n_0 - n_1 + n_2.
\end{equation}

No caso da esfera $\Sp^2$ e na Figura \ref{triangsphere}, e pela triangula\c c\~ao do cubo, temos 
$$n_0 - n_1 + n_2 = 8 - 18 + 12 =2.$$
Pela triangula\c c\~ao do  tetraedro, temos 
$$n_0 - n_1 + n_2 = 4 -6 + 4 =2.$$

H\'a v\'arias demonstra\c c\~oes algumas discut\'iveis, 
em particular por Descartes (1639) e Cauchy (1811), de que, no caso da esfera,  
a caracter\'\i stica de Euler-Poincar\'e
n\~ao depende da triangula\c c\~ao $K$ escolhida (veja os 
artigos de Elon Lima \cite{Li1,Li2}, tamb\'em \cite{BN} e o site web \cite{Epp}). 
Assim  podemos denotar por  $\chi(\Sp^2)$ e chamar de 
caracter\'\i stica de Euler-Poincar\'e da esfera a quantitade $\chi(K)$ para qualquer 
triangula\c c\~ao da esfera. 
Assim temos 
$$
\chi(\Sp^2)=n_0 - n_1 + n_2=2, 
$$
qualquer que seja a triangula\c c\~ao da esfera.
\\
A invari\^ancia da triangula\c c\~ao vale para qualquer superf\'icie compacta $\Su$ e podemos 
definir $\chi(\Su)$ como sendo $\chi(K)$ para qualquer triangula\c c\~ao da superf\'icie. 

\begin{center}
\begin{figure}

\begin{tikzpicture}[scale=1.05]

 \draw[very thick] (0,0) -- (4,0) -- (4,1) -- (0,1) -- (0,0);
\draw[very thick] (1,-1) -- (2,-1) -- (2,2) -- (1,2) -- (1,-1);
\draw[very thick] (3,0) -- (3,1) ;
\draw[very thick] (0,0) -- (2,2) ;
\draw[very thick] (1,0) -- (2,1) ;
\draw[very thick] (1,-1) -- (3,1) ;
\draw[very thick] (3,0) -- (4,1) ;

 \draw[->,very thick] (0,0) -- (0,0.5);
 \node at (0,0.5) [left]{$d$};
  \draw[->,very thick] (0,1) -- (0.5,1);
  \node at (0.5,1) [above] {$c$};
  \draw[->,very thick] (1,2) -- (1,1.5);
  \node at (1,1.5) [left] {$c$};
  \draw[->,very thick] (2,2) -- (1.5,2);
  \node at (1.5,2) [above] {$b$};
  \draw[->,very thick] (2,1) -- (2,1.5);
  \node at ((2,1.5) [right] {$a$};
   \draw[->,very thick] (2,1) -- (2.5,1);
  \node at (2.5,1) [above] {$a$};
  \draw[->,very thick] (3,1) -- (3.5,1);
  \node at (3.5,1) [above] {$b$};
  \draw[->,very thick] (4,0) -- (4,0.5);
  \node at (4,0.5) [right] {$d$};
  \draw[->,very thick] (3,0) -- (3.5,0);
  \node at (3.5,0) [below] {$f$};    
     \draw[->,very thick] (2,0) -- (2.5,0);
  \node at (2.5,0) [below] {$g$};
  \draw[->,very thick] (2,0) -- (2,-0.5);
  \node at (2,-0.5) [right] {$g$};
  \draw[->,very thick] (2,-1) -- (1.5,-1);
   \node at (1.5,-1) [below] {$f$};
    \draw[->,very thick] (1,-1) -- (1,-0.5);
   \node at (1,-0.5) [left] {$e$};
\draw[->,very thick] (0,0) -- (0.5,0);
   \node at (0.5,0) [below] {$e$};

  \draw[->,very thick] (5,0.5) -- (6,0.5);

\draw[very thick] (7,0) -- (8,0) -- (8,1) -- (7,1) -- (7,0) -- (8,1);
\draw[very thick] (7,1) -- (7.4,1.2) -- (8.4,1.2)  -- (8.4,0.2) -- (8,0) -- (8.4,1.2) -- (8,1);
 \draw[very thick] (7,1) -- (8.4,1.2);
 \draw[dotted,very thick] (7,0) -- (7.4,0.2) -- (7.4,1.2) ;
 \draw[dotted,very thick] (7.4,0.2) -- (8.4,0.2) ;

\draw[very thick] (0.5,-3) -- (2.5,-3) -- (1.5,-4.73) -- (0.5,-3);
\draw[very thick] (1.5,-3) -- (1,-3.86) -- (2,-3.86) -- (1.5,-3);

    \draw[->,very thick] (0.5,-3) -- (1,-3);
   \node at (1,-3) [above] {$c$};    
   \draw[->,very thick] (2.5,-3) -- (2,-3);
   \node at (2,-3) [above] {$c$};
  \draw[->,very thick] (0.5,-3) -- (0.75,-3.43);
   \node at (0.75,-3.43) [left] {$a$};    
   \draw[->,very thick] (1.5,-4.73)-- (1.25,-4.3);
   \node at (1.25,-4.3) [left] {$a$};
  \draw[->,very thick] (1.5,-4.73)-- (1.75,-4.3);
   \node at (1.75,-4.3) [right] {$b$};    
   \draw[->,very thick] (2.5,-3) --  (2.25,-3.43);
   \node at (2.25,-3.43) [right] {$b$};

  \draw[->,very thick] (4.5,-3.5) -- (5.5,-3.5);
  
\draw[very thick] (7.4,-3) -- (8,-4) -- (7,-4.4) -- (7.4,-3) -- (6.4,-4) --(7,-4.4);
   \draw[dotted,very thick] (6.4,-4) -- (8,-4) ;
   
 \draw [very thick](11.5, -1.7) circle (0.75);
 
 \draw[dotted,very thick] (12.25, -1.7) arc (0:180:0.75 and 0.3);
\draw[very thick]  (10.75, -1.7) arc (180: 360:0.75 and 0.3);
   
    \draw [->,very thick] (9.8,0.4) arc (60:25:2.5);
     \draw [->,very thick] (9.2,-3.6) arc (270:320:2);
     
   \node at (10,-0.2) [right] {$h$};
      \node at (9.8,-3.2) [right] {$h$};    
\end{tikzpicture}
\caption{Triangula\c c\~oes da esfera.  Nas representa\c c\~oes  planares
identificam-se os  segmentos de mesmo nome com a mesma orienta\c c\~ao.}\label{triangsphere}
\end{figure}
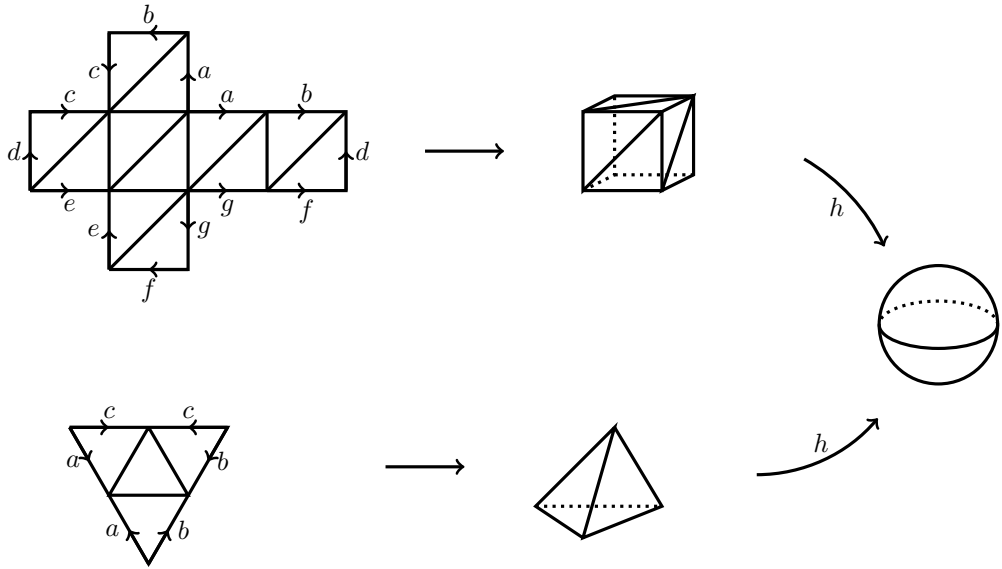
\end{center}

A caracter\'\i stica de Euler-Poincar\'e do toro vale $0$. 
Em geral, a caracter\'istica de Euler-Poincar\'e de uma superf\'icie compacta, sem borda e orient\'avel 
\'e $2 -2g$, onde  $g$ \'e o  g\^enero da superf\'icie. 
Lembramos que o {\it g\^enero} $g$ de uma superf\'icie (orient\'avel ou n\~ao) 
\'e o n\'umero m\'aximo de circunfer\^encias que se pode desenhar sobre  a superf\'icie sem desconect\'a-la.
 Por exemplo, o g\^enero de uma  esfera \'e $0$, pois  qualquer circunfer\^encia desenhada sobre 
 a esfera a desconecta. O g\^enero de um toro \'e $1$, pois \'e poss\'ivel desenhar uma circunfer\^encia 
 sobre o toro sem desconect\'a-lo, mas qualquer 
segunda circunfer\^encia desenhada sobre o toro o desconecta (figura \ref{genero}).

\begin{center}
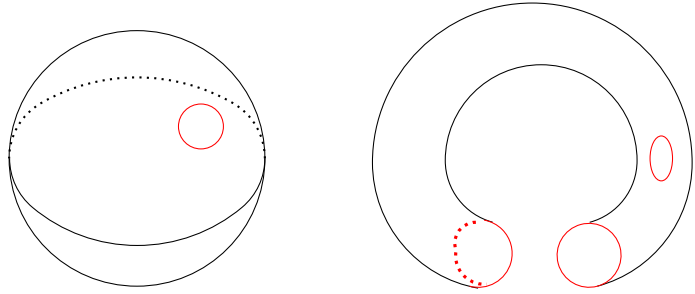
\begin{figure}[h!]
\begin{tikzpicture} [scale=0.85]

\draw (-7,0) circle (2cm);
\draw[dotted, thick]  (-5cm,-0cm) arc (0:60:1) arc (60:90:3) arc (90:120:3) arc (120:180:1);
\draw  (-5cm,-0cm) arc (0:-50:1) arc (-50:-90:2.55) arc (-90:-130:2.55) arc (-130:-180:1);
\draw[red] (-6, 0.5) ellipse (10 pt and 10 pt);

\draw (0.2cm,-2cm) arc (-75:0:2.0) arc (0:180:2.5) arc (0.2:80:-2);
\draw   (0.08 cm,-1cm) arc (-75:0:1) arc (0:180:1.5) arc (0.1:75:-1);
\draw [red](0.2cm,-2cm) arc (-75:285:0.5);
\draw [red] (-1.8cm,-2cm) arc (-105:80:0.53);
\draw[red, very thick, dotted] (-1.7cm,-1cm) arc (90:200:0.33) arc (170:275:0.45);
\draw[red] (1.2,0) ellipse (5pt and 10pt);
\end{tikzpicture} 
\caption{G\^enero da esfera e do toro.} \label{genero}
\end{figure}
\end{center}

O plano projetivo ${\mathbb{P}}^2$ \'e o conjunto de todas  as retas do espa{\c c}o euclidiano $\R^3$ passando pela origem. Uma maneira f\'acil para representar o espa{\c c}o projetivo \'e considerar em $\R^3$ a esfera $\Sp^2$ centrada na origem e de raio 1. Consideramos tr\^es partes na esfera $\Sp^2$: as 
semiesferas (abertas) norte e sul e o equador (veja figura \ref{figura17}).

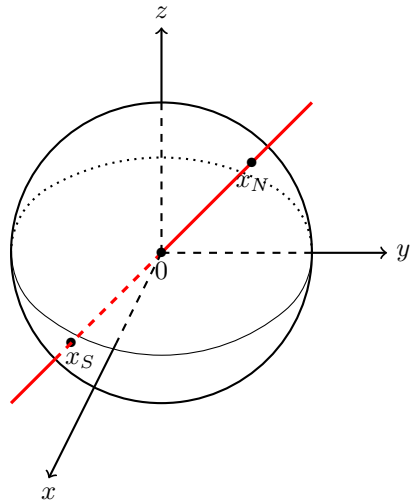
\begin{figure}[h]
\begin{tikzpicture} 
\draw[thick] (2,2) circle (2cm);
\draw[dotted, thick]  (4cm, 2cm) arc (0:60:1) arc (60:90:3) arc (90:120:3) arc (120:180:1);
\draw  (4cm,2cm) arc (0:-50:1) arc (-50:-90:2.55) arc (-90:-130:2.55) arc (-130:-180:1);

\draw[thick, dashed] (2,2) --(2,4);
\draw[thick, ->] (2,4) -- (2,5);
\node at (2, 5)[above] {$z$};

\draw[thick, dashed] (2,2) -- (4,2);
\draw[thick, ->] (4,2) -- (5,2);
\node at (5,2)[right] {$y$};

\draw[very thick,red] (2,2) -- (4,4);
\node at (3.2,3.2) {$\bullet$};
\node at (3.22,3.15)[below] {$x_N$};
\node at (0.8,0.8) {$\bullet$};
\node at (0.92,0.8)[below] {$x_S$};
\draw[very thick, red,dashed] (2,2) -- (0.6,0.6);
\draw[very thick,red] (0.6,0.6) -- (0,0);

\draw[thick, dashed] (2,2) -- (1.4, 0.8);
\draw[->, thick] (1.4, 0.8) -- (0.5, -1);

\node at  (0.5, -1)[below] {$x$};

\node at (2,2) {$\bullet$};
\node at (2,2)[below] {0};

\end{tikzpicture} 
\caption{Representa{\c c}\~ao do espa\c co projetivo ${\mathbb{P}}^2$.}\label{figura17}
\end{figure}

Cada reta de $\R^3$ passando pela origem  e n\~ao contida no plano $(0xy)$ 
encontra as semiesferas norte em um ponto $x_N$ e sul em um ponto $x_S$.
O ponto $x_N$ pode ser considerado como um representante da reta. 
As retas contidas no plano $(0xy)$ encontram o equador em dois pontos diametralmente 
opostos, que temos que identificar para ter um representante s\'o. 
Assim obtemos uma representa{\c c}\~ao do plano projetivo 
${\mathbb{P}}^2$, como sendo a semiesfera norte 
(ou seja um disco) com identifica\c c\~ao de pontos diametralmente opostos na sua  borda. 

A caracter\'\i stica de Euler-Poincar\'e do plano projetivo vale $+1$, da  garrafa de Klein vale $0$. 
Em geral, a caracter\'\i stica de Euler-Poincar\'e de uma superf\'icie compacta e sem borda, n\~ao 
orient\'avel, de g\^enero $k$, vale $2-k$. 

Lembramos tamb\'em que uma superf\'icie compacta e sem borda pode ser mergulhada em $\R^3$ se, e somente se, ela \'e orient\'avel.


\section{\'Indice de um campo de vetores no plano}
\subsection{\'Indice de uma curva plana fechada }

Seja $\Gamma$ uma curva plana fechada e orientada no plano orientado $\R^2$.  
Seja $\alpha$ um ponto n\~ao situada na curva $\Gamma$. Consideramos um pequena circunfer\^encia $C_\alpha$ 
de centro $\alpha$ tal que $\Gamma \cap C_\alpha = \emptyset$, orientada com orienta\c c\~ao 
induzida por $\R^2$. 

Seja $y$ um ponto na curva $\Gamma$ e $\alpha(y)$ o ponto interse\c c\~ao $\alpha(y)= [\alpha,y] \cap C_\alpha$. 
Quando $y$ caminha sobre a curva $\Gamma$ e descreve toda a curva no sentido (positivo) da curva, 
o ponto $\alpha(y)$ faz voltas sobre  $C_a$ no sentido positivo ou negativo. Seja $p_\alpha$ o n\'umero de voltas no sentido positivo e $q_\alpha$ o n\'umero de voltas no sentido negativo.
\begin{definition}
O \'indice $I(\Gamma, \alpha)$ da curva $a$ relativamente ao ponto $\alpha$ \'e a diferen\c ca
$$I(\Gamma, \alpha) = p_\alpha - q_\alpha.$$
\end{definition}

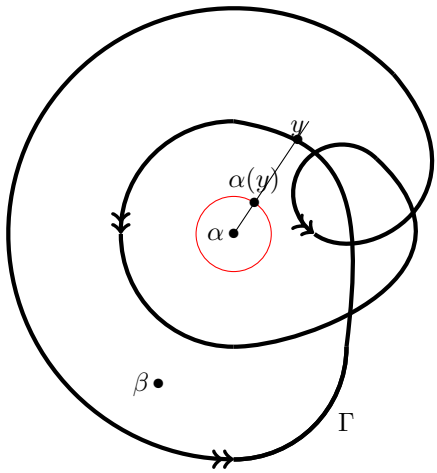
\begin{figure}[h]
\begin{center}
\begin{tikzpicture}[scale=0.5]

  \draw [->>,ultra thick] (-6,3) arc (90:180:3);
    \draw [ultra thick] (-9,0) arc (180:270:3);
   \draw[->>,ultra thick]  (-1.75,4.25) arc (45:270:6);
   \draw [ultra thick] (-6,-6) arc (-90:0:3);
    \draw [ultra thick] (-6,-6) arc (-90:0:3);
    \draw [->>,ultra thick] (-2.2,2) arc (45:235:1.3);
    
   \draw[ultra thick] 
  (-6,3) 
    .. controls (-3,2.5) and (-2.5,1) .. 
  (-3,-3); 
  
   \draw[ultra thick] 
  (-2.2,2) 
    .. controls (1,-1) and (-4,-3) .. 
  (-6,-3); 

 \draw[ultra thick] 
  (-1.75,4.25) 
    .. controls (1,1) and (-2.5,-1) .. 
  (-3.85,0); 

 \draw[red] (-6, 0) circle (1);
 \draw (-6,0) -- (-4,3);
 \node at (-6,0) [below,left] {$\alpha$};
  \node at (-8,-4) [below,left]  {$\beta$};
     \node at (-5.45,0.82) [above]{$\alpha(y)$};
 \node at (-3,-5){$\Gamma$};
 \node at (-4.3,2.8){$y$};
 
  \node at (-4.3,2.5){$\bullet$};
    \node at (-5.45,0.82){$\bullet$};
  \node at (-6,0)   {$\bullet$};
   \node at (-8,-4)   {$\bullet$};
   
\end{tikzpicture}
\caption{$I(\Gamma,\alpha)=+2$; \quad $I(\Gamma,\beta)=-1$} 
\end{center}
\end{figure}

O n\'umero $ p_\alpha - q_\alpha$ n\~ao depende do ponto de partida $y$ sobre a curva $\Gamma$. Ele n\~ao depende tamb\'em do ponto $\alpha$ situado na mesma componente conexa aberta de $\R^2 \setminus \Gamma$. 
Mas, em uma outra componente conexa aberta de $\R^2 \setminus \Gamma$, o \'indice  $I(\Gamma, \beta)$   num ponto $\beta$ pertencente a esta componente 
pode ser diferente de $I(\Gamma, \alpha)$. 

\subsection{\'Indice de um campo de vetores planar em um ponto singular isolado. }

Consideramos agora uma circunfer\^encia $\Sp^1_a$ de centro $a$ em $\R^2$ e um campo cont\'inuo de vetores
 $v(x)$ definido sobre o disco fechado $D_a\subset \R^2$ cuja $\Sp^1_a$ \'e borda. 
Suponhamos que 
a \'unica singularidade  de $v$ no disco $D_a$ seja no ponto $a$. 
Isto \'e, $a$ \'e o \'unico ponto do disco $D_a$ tal que $v(a)$ = 0.

\begin{figure}[h]
\begin{tikzpicture} [scale=0.45]

 \draw[red] (0, 0) circle (6);
\node at (0,0) {$\bullet$};
\node at (0,0)[above, right] {$a$};

\node at (6,0) {$\bullet$};
\node at (6,0) [left] {\small $1$};

\path (1*10:6) coordinate (A1);
\node at (A1) {$\bullet$};
\node at (A1) [left] {\small $2$};
 \path (2*15:6) coordinate (A2);
 \node at (A2) {$\bullet$};
 \node at (A2) [left] {\small $3$};
  \path (7*7.5:6) coordinate (A4);
 \node at (A4) {$\bullet$};
 \node at (A4) [right] {\small $4$};
  \path (5*15:6) coordinate (A5);
 \node at (A5) {$\bullet$};
 \node at (A5) [above] {\small $5$};
   \path (6*15:6) coordinate (A6);
 \node at (A6) {$\bullet$};
\node at (A6) [above] {\small $6$};
    
  \draw[->] (6,0) -- (9,0); 
  \begin{scope}[shift={(A1)},rotate=60]; 
  \draw[->] (A1) -- (2.7,-0.1);
  \end{scope}
   \begin{scope}[shift={(A2)},rotate=100];
    \draw[-> ] (A2) -- (2.8,0);
  \end{scope}
   \begin{scope}[shift={(A4)},rotate=180]; 
    \draw[->] (A4) -- (2.8,0);
  \end{scope}
 \begin{scope}[shift={(A5)},rotate=230];
    \draw[->] (A5) -- (2.8,0);
  \end{scope}
  \begin{scope}[shift={(A6)},rotate=270];
    \draw[->] (A6) -- (2.8,0);
  \end{scope}


 \path (5*22.5:6) coordinate (A7);
 \node at (A7) {$\bullet$};
 \node at (A7) [above] {\small $7$};
 \path (6*22.5:6) coordinate (A8);
 \node at (A8) {$\bullet$};
  \node at (A8) [left] {\small $8$};
  \path (7*22.5:6) coordinate (A9);
 \node at (A9) {$\bullet$};
   \node at (A9) [left] {\small $9$};
  \path (8*22.5:6) coordinate (A10);
 \node at (A10) {$\bullet$};
    \node at (A10) [left] {\small $10$};
      \path (8.5*22.5:6) coordinate (A11);
 \node at (A11) {$\bullet$};
    \node at (A11) [left] {\small $11$};
  \path (9*22.5:6) coordinate (A12);
 \node at (A12) {$\bullet$};
   \node at (A12) [left] {\small $12$};
\path (9.7*22.5:6) coordinate (A13);
 \node at (A13) {$\bullet$};
  \node at (A13) [left] {\small $13$};
  \path (10.6*22.5:6) coordinate (A14);
 \node at (A14) {$\bullet$};
 \node at (A14) [below] {\small $14$};
  \path (11.3*22.5:6) coordinate (A15);
 \node at (A15) {$\bullet$};
  \node at (A15) [below] {\small $15$};
  
 \begin{scope}[shift={(A7)},rotate=305];
  \draw[-> ] (A7) -- (3.5,0);
  \end{scope}
   \begin{scope}[shift={(A8)},rotate=350];
    \draw[-> ] (A8) -- (5,0);
  \end{scope}
  \begin{scope}[shift={(A9)},rotate=390];
    \draw[-> ] (A9) -- (4.3,0);
  \end{scope}
 \begin{scope}[shift={(A10)},rotate=410];
    \draw[->] (A10) -- (2.5,0);
  \end{scope}
 \begin{scope}[shift={(A11)},rotate=390];
    \draw[-> ] (A11) -- (1.5,0);
  \end{scope}
  \begin{scope}[shift={(A12)},rotate=370];
    \draw[->] (A12) -- (1.7,0);
  \end{scope}
  \begin{scope}[shift={(A13)},rotate=360];
  \draw[->] (A13) -- (3.8,0);
  \end{scope}
   \begin{scope}[shift={(A14)},rotate=405];
    \draw[->] (A14) -- (4.7,0);
  \end{scope}
 \begin{scope}[shift={(A15)},rotate=445];
    \draw[-> ] (A15) -- (5.7,0);
  \end{scope}
  

\path (18*15:6) coordinate (A16);
\node at (A16) {$\bullet$};
 \node at (A16) [below] {\small $16$};
 \path (19*15:6) coordinate (A17);
\node at (A17) {$\bullet$};
 \node at (A17) [below] {\small $17$};
  \path (20.5*15:6) coordinate (A18);
 \node at (A18) {$\bullet$};
  \node at (A18) [above] {\small $18$};
  \path (21.7*15:6) coordinate (A19);
 \node at (A19) {$\bullet$};
   \node at (A19) [left] {\small $19$};
    \path (23*15:6) coordinate (A20);
 \node at (A20) {$\bullet$};
  \node at (A20) [left] {\small $20$};
 
\begin{scope}[shift={(A16)},rotate=120];
   \draw[->] (A16) -- (5.2,0);
 \end{scope}
  \begin{scope}[shift={(A17)},rotate=150];
  \draw[->] (A17) -- (5.5,0);
  \end{scope}
   \begin{scope}[shift={(A18)},rotate=220];
    \draw[-> ] (A18) -- (5,0);
  \end{scope}
  \begin{scope}[shift={(A19)},rotate=280];
    \draw[->] (A19) -- (4.5,0);
      \end{scope}
      \begin{scope}[shift={(A20)},rotate=320];
    \draw[-> ] (A20) -- (3.5,0);
  \end{scope}

\draw[->] (10,1) -- (12,1);
\node at (11,1) [above] {$\widetilde\gamma_v$};


\node at (20,0) {$\bullet$};
\node at (20,0)[above,left] {$0$};

  \draw [->>,ultra thick] (20,3) arc (90:135:3);
   \draw [ultra thick] (17.88,2.282) arc (135:180:3);
   
    \draw [ultra thick] (17,0) arc (180:270:3);
   \draw[->>,ultra thick]  (24.25,4.25) arc (45:270:6);
   \draw [ultra thick] (20,-6) arc (-90:0:3);
    \draw [ultra thick] (23.8,2) arc (45:235:1.3);

  \draw[ultra thick] 
  (20,3) 
    .. controls (23,3) and (23.5,1) .. 
  (23,-3); 
  
   \draw[ultra thick] 
  (23.8,2) 
    .. controls (27,-1) and (22,-3) .. 
  (20,-3); 

 \draw[ultra thick] 
  (24.25,4.25) 
    .. controls (27,1) and (23.5,-1) .. 
  (22.15,0); 

\node at (23.17,0)   {$\bullet$};
\node at (23.1,0.4) [right]  {$\widetilde 1$};
\node at (21.5,2.78)   {$\bullet$};
\node at (21.5,2.78) [above]  {$\widetilde 2$};
\node at (19,2.85)   {$\bullet$};
\node at (19,2.85)  [above]  {$\widetilde 3$};
\node at (17,0)   {$\bullet$};
\node at (17,0) [left]   {$\widetilde 4$};
\node at (18,-2.25)   {$\bullet$};
\node at (18,-2.25) [below]  {$\widetilde 5$};
\node at (20,-3)   {$\bullet$};
\node at (20,-3)  [below] {$\widetilde 6$};
\node at (22,-2.63)   {$\bullet$};
\node at (22,-2.63)   [below]  {$\widetilde 7$};
\node at (24.61,-0.8)   {$\bullet$};
\node at (24.61,-0.8)  [right] {$\widetilde 8$};
\node at (23.6,2.2)   {$\bullet$};
\node at (23.6,2.2) [above]  {$\widetilde 9$};
\node at (22,2.051)   {$\bullet$};
\node at (22,2) [above,left]  {$\widetilde {10}$};
\node at (21.625,0.8)   {$\bullet$};
\node at (21.625,0.8) [left]  {$\widetilde {11}$};
\node at (22,0.15)   {$\bullet$};
\node at (22,0.15)  [below] {$\widetilde {12}$};
\node at (24.2,0.12)   {$\bullet$};
\node at (24.2,0.12)  [above] {$\widetilde {13}$};
\node at (24.45,4)   {$\bullet$};
\node at (24.45,4)  [right] {$\widetilde {14}$};
\node at (20,6)   {$\bullet$};
\node at (20,6)  [above] {$\widetilde {15}$};
\node at (16.6,4.92)   {$\bullet$};
\node at (16.6,4.92) [above]  {$\widetilde {16}$};
\node at (14.65,2.7)   {$\bullet$};
\node at (14.65,2.7) [right]  {$\widetilde {17}$};
\node at (15.37,-3.8)   {$\bullet$};
\node at (15.37,-3.8)  [right] {$\widetilde {18}$};
\node at (21,-5.82)   {$\bullet$};
\node at (21,-5.82)  [above] {$\widetilde {19}$};
\node at (23,-3)   {$\bullet$};
\node at (23,-3) [right]  {$\widetilde {20}$};

\end{tikzpicture}
\caption{O campo de vetores $v$ e a imagem de $\Sp^1_a$ por 
$\widetilde \gamma_v$. Temos $I(v,0)=+2$. Para todo ponto $\small i$ de $\Sp^1_a$, o ponto 
$\widetilde \gamma_v ({\small i})$ \'e denotado por $\widetilde \imath$. 
O vetor $
\widearrow{0 \, \widetilde{\imath}}$ \'e  paralelo ao vetor $v({\small i})$. }
\end{figure}

Podemos definir uma aplica\c c\~ao 
$\widetilde \gamma_v$ de $\Sp^1_a$ em uma outra  c\'opia de $\R^2$ denotado por $\widetilde \R^2$
tal que para cada $x\in \Sp^1_a$, o ponto $\widetilde \gamma_v(x)$ seja a meta do vetor
$\widetilde w(x)$ paralelo \`a $v(x)$ e de origem \`a origem $0$ de $\widetilde \R^2$.

A imagem de $\Sp^1_a$ por $\widetilde \gamma_v$ \'e uma curva $\widetilde \Gamma = \widetilde \gamma_v  
(\Sp^1_a)$. Observamos que, como $v$ n\~ao se anula sobre $\Sp^1_a$, a curva 
$\widetilde \Gamma$ n\~ao cont\'em a origem $0$ de $\widetilde \R^2$.

\begin{definition}\label{naosei}
O \'indice $I(v,a)$ \'e definido como o \'indice 
da curva $\widetilde \Gamma = \widetilde \gamma_v 
(\Sp^1_a)$ relativamente ao ponto $0$,  isto \'e, $I(v, a) = I(\widetilde \Gamma,0)$.
\end{definition} 

\section{Grau de uma aplica\c c\~ao suave}
\subsection{Caso de curvas} \label{tangcirc}
\subsubsection{Fibrado tangente a uma curva suave e fechada}

Seja $\mathcal{C}$ uma curva suave e fechada em $\R^2$. 
 Em cada ponto $x$ de $\mathcal{C}$, temos a reta tangente \`a $\mathcal{C}$ no ponto $x$, 
que denotamos por $T_x\mathcal{C}$.  Vemos que $T_x\mathcal{C}$ \'e um espa{\c c}o vetorial de dimens\~ao 1 sobre $\R$. 

Um vetor tangente \`a curva suave  $\mathcal{C}$ no ponto $x$ \'e um elemento $v(x)$ de $T_x\mathcal{C}$. 
Podemos definir um vetor tangente como  a classe de equival\^encia de curvas diferenci\'aveis 
$c : \, \, ]-1,+1[ \to \mathcal C$  desenhadas sobre $\mathcal{C}$ e tais que $c(0)=x$. 
Temos $v(x) = c'(0)$. 

Denotamos $T \mathcal{C} = \bigcup_{x \in \mathcal{C}} T_x\mathcal{C}$ e chamamos $T\mathcal{C}$ 
de {\it espa{\c c}o fibrado tangente \`a} $\mathcal{C}$. Cada $T_x{\mathcal{C}}$ se chama {\it fibra} de $T\mathcal{C}$ no ponto $x$ e \'e  homeomorfa a $\R$. Chamamos $\R$ de {\it ``fibra tipo''}. 

O fibrado $T \mathcal{C}$ \'e localmente trivial, isso significa que para todo ponto $x$ de $\mathcal{C}$, existe uma vizinhan{\c c}a $U_x$ de $x$ em $\mathcal{C}$ tal que 
$$T \mathcal{C} \vert_{U_x} \cong U_x \times \R,$$
 onde $T \mathcal{C} \vert_{U_x} $ \'e a restri{\c c}\~ao de $T \mathcal{C}$ a $U_x$. 

De fato, o fibrado tangente \`a $\mathcal{C}$ \'e 
globalmente trivial ($T\mathcal{C} \cong \mathcal{C} \times \R$), mas veremos que esta propriedade 
  n\~ao \'e verdadeira em geral. Por exemplo, o fibrado tangente \`a esfera ${\mathbb{S}}^2$
n\~ao \'e globalmente trivial (veja a Se\c c\~ao (\ref{2.2})).

O fibrado $T \mathcal{C}$ tem uma vizualiza{\c c}\~ao natural no plano $\R^2$ 
mas esta vizualiza{\c c}\~ao n\~ao \'e  conveniente. Uma maneira mais conveniente \'e 
olhar o plano $\R^2$ como plano horizontal em $\R^3$ e girar todas as retas tangentes  em
 $+ 90^{\rm o}$, verticalmente em rela{\c c}\~ao ao plano $\R^2$ 
(veja Figura \ref{figura22}).

\begin{figure}[h]
\begin{tikzpicture} 


\draw (0,0) circle (2cm);

\draw[->] (-4,2) -- (3,2);
\node at (0,2) {$\bullet$};
\node at (0,2)[above] {$y$};
\draw[->, red, thick] (0,2) -- (2,2);
\node at (2,2)[above, red] {$v(y)$};

\draw[->] (-3,0.5) -- (1,3);
\node at (1.4,0.8)[above] {$\mathcal{C}$};

\draw[->] (-1.2,-2.5) -- (-3.35,3);
\node at (-1.95,-0.5) {$\bullet$};
\node at (-1.95,-0.5)[left] {$x$};
\draw[->, red, thick] (-1.95,-0.5) -- (-2.7,1.4);
\node at (-2.7,1.4)[left, red] {$v(x)$};
\node at (-1.4,-2)[left] {$T_x\mathcal{C}$};

\draw[<-] (0.6,-2.5) -- (3.2,0.8);
\node at (1.55,-1.3) {$\bullet$};
\node at (1.55,-1.3) [right] {$z$};
\draw[->, red, thick] (1.55,-1.3) -- (2.6,0);
\node at (2.5,0)[right, red] {$v(z)$};

\node at (3.7, 0) {$\rightarrow$};

\draw (4, -1.5) -- (5,1.5) -- (5.5, 1.5);
\draw[dotted, thick] (5.5, 1.5) -- (9.5, 1.5);
\draw (9.5, 1.5) --  (11, 1.5) -- (10, -1.5) -- (4, -1.5);
\draw[dotted,very thick] (9.5, 0.2) arc (0:185:2cm and 1cm);
\draw[very thick]  (5.5, 0) arc (-138: -91:2.6) arc (-91:-46: 2.9);

\node at (5.5, 0) {$\bullet$}; 
\node at (6, -0.4614) [left] {$x$}; ;\node at (6, -0.4614) {$\bullet$};
\node at (8, -0.8) {$\bullet$};
\node at (8.9, -0.45)[below] {$z$}; \node at (9, -0.45) {$\bullet$};
\node at (9.5, 0) {$\bullet$};

\node at (7.3, 1.2) {$\bullet$};
\node at (7.18, 1.2)[below] {$y$};

\draw[->,thick] (5.5, 0) -- (5.5, 3.5);
\draw[dotted, thick] (5.5, 0) -- (5.5, -1.5);
\draw[-,thick] (5.5, -1.5) -- (5.5, -3);

\draw[->,thick] (6, -0.4614) -- (6, 3.5);
\draw[dotted, thick] (6, -0.4614) -- (6, -1.5);
\draw[-,thick] (6, -1.5) -- (6, -3);
\draw[->, red, thick] (6, -0.4614) -- (6, 1.2);
\node at  (6, 1.2) [right, red] {$v(x)$};

\draw[->, thick] (7.3, 2.8) -- (7.3, 3.5);
\draw[dotted, thick] (7.3, 1.2) -- (7.3, -1.5);
\draw[-,thick] (7.3, -1.5) -- (7.3, -3);
\draw[->, red, thick] (7.3, 1.2) -- (7.3, 2.8);
\node at  (7.3, 2.8) [left, red] {$v(y)$};

\draw[->,thick] (8, -0.8) -- (8, 3.5);
\draw[dotted, thick] (8, -0.8) -- (8, -1.5);
\draw[-,thick] (8, -1.5) -- (8, -3);

\draw[->,thick] (9, -0.45) -- (9, 3.5);
\draw[-,thick] (9, -2.5) -- (9, -3);
\draw[-,red,dotted, thick] (9,-0.45) -- (9, -1.5);
\draw[->, red, thick] (9,-1.5) -- (9, -2.5);
\node at  (9, -2.5)  [left, red] {$v(z)$};

\draw[->,thick] (9.5, 0) -- (9.5, 3.5);
\draw[dotted, thick] (9.5, 3) -- (9.5, -1.5);
\draw[-,thick] (9.5, -1.5) -- (9.5, -3);

\end{tikzpicture} 
\caption{O fibrado $T \mathcal{C}$, aqui tomamos para $\mathcal{C}$ 
a circunfer\^encia $\Sp^1$.}\label{figura22}
\end{figure}

Obtemos assim uma representa\c c\~ao  de $T \mathcal{C}$ como um cilindro de base $\mathcal{C}$. 
A dire{\c c}\~ao positiva de  $T_x \mathcal{C}$ \'e para acima e a dire{\c c}\~ao  negativa \'e para abaixo. 
Assim, \'e poss\'ivel representar o vetor $v(x)$ positivo, negativo ou nulo. 

Seja $U$ um aberto de $\mathcal{C}$ (um segmento aberto), um campo $v$ de vetores tangentes
 \`a $\mathcal{C}$ ao longo de $U$ \'e uma aplica{\c c}\~ao cont\'inua $v : U \to T\mathcal{C} \vert_U$ tal que $v(x)$ pertence \`a $T_x(\mathcal{C})$,  para cada $x \in U$.

Uma singularidade do campo $v$ de vetores tangentes \`a $\mathcal{C}$ \'e um ponto $a$ tal que $v(a) = 0$. Uma singularidade $a$ ser\'a chamada de {\it singularidade isolada} se existir uma vizinhan{\c c}a 
$V_a$ do ponto $a$,   em $\mathcal{C}$, tal que $v(a) = 0$ e $v(x) \neq 0$, para 
$x \in V_a \setminus \{ a \}$. 
\medskip

\subsubsection{Pontos e valores cr\'iticos, grau} \label{grau} \hfill

Seja $f: \mathcal{C}  \to \Sp^1$ uma aplica\c c\~ao suave. 
Damos \`a  curva fechada $\mathcal C$ a orienta\c c\~ao induzida pela orienta\c c\~ao de $\R^2$. 
Consideramos a mesma orienta\c c\~ao no secundo exemplar $\widetilde\R^2$, assim temos para $\Sp^1$
 a orienta\c c\~ao induzida de $\widetilde\R^2$.
 
A  {\it derivada}  da aplica\c c\~ao suave 
$f: \mathcal{C}  \to \Sp^1$  em um ponto $x$ da curva $\mathcal{C}$, chamada de 
{\it aplica\c c\~ao linear tangente} e denotada por
\begin{equation} \label{derivee}
df_x : T_x\mathcal{C} \to T_{f(x)}\Sp^1, \qquad \text{ ou seja, } \qquad df_x :{\mathbb R} \to {\mathbb R}, 
\end{equation}
\'e definida por 
\begin{equation} \label{derivee1}
df_x(v) = df_x(c'(0)) = (f \circ c)'(0)
\end{equation}
para toda curva $c : \, \, ]-1,+1[ \to \mathcal C$ tal que $c(0)=x$ e $v(x) = c'(0)$. 
De fato, $f\circ c$ \'e uma curva sobre  $\Sp^1$ 
tal que $(f\circ c)(0) = f(x)$.  Neste caso,  $(f \circ c)'(0)$ \'e um vetor tangente 
\`a  $\Sp^1$ no ponto $f(x)$. 

Vamos considerar agora a no\c c\~ao de valores cr\'iticos de uma aplica\c c\~ao suave 
$f: \mathcal{C}  \to \Sp^1$. Um ponto $a\in \mathcal{C} $ tal que 
$df_a : T_a\mathcal{C} \to T_{f(a)}\Sp^1$
tem o posto 0 \'e chamado de {\it ponto cr\'itico} de $f$.

Por defini\c c\~ao, um  ponto regular  da aplica\c c\~ao $f$ \'e um ponto  $x\in {\mathcal C} $ que n\~ao \'e cr\'itico. Em um  ponto  regular, $df_x : T_x{\mathcal C}  \to T_{f(x)} \Sp^1$ \'e um isomorfismo de espa\c cos vetoriais 
orientados e definimos o {\it sinal} de $df_x$ \'e $+1$ ou $-1$ se $df_x$ preserva ou n\~ao 
a orienta\c c\~ao, respectivamente. 

As imagens de pontos cr\'iticos s\~ao chamadas de {\it valores cr\'iticos} e os demais pontos s\~ao chamados {\it valores regulares} da aplica\c c\~ao.
Ent\~ao, um ponto $y\in \Sp^1$ \'e um valor regular  de $f$ se $f^{-1}(y) = \emptyset$ ou 
todos os pontos $x\in f^{-1}(y)$ 
s\~ao pontos regulares de $f$. Em um valor regular $y$ de $f$, podemos considerar o n\'umero inteiro 
$$ {\rm grau}(f;y) = \sum_{x\in f^{-1}(y)} {\rm sign} \; df_x,$$
onde ${\rm sign} \; df_x$ \'e o sinal de $df_x$ no ponto $x \in f^{-1}(y)$.

\begin{example}\label{exemplo1}
Consideremos o exemplo da aplica\c c\~ao $f : \Sp^1 \to \Sp^1$ dada pelo grafo da Figura \ref{aplicercle2}. 

\begin{center}
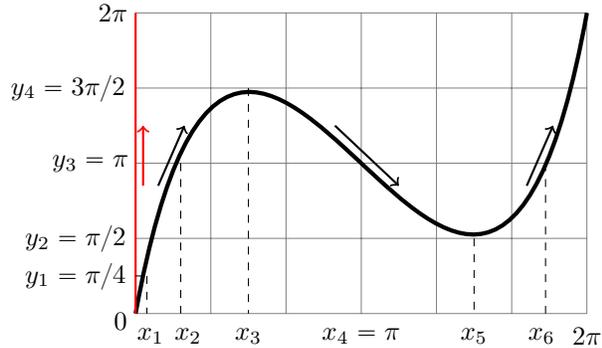
\begin{figure}[H]
\vglue -3.5 truecm
\begin{tikzpicture}[baseline]
\draw[help lines] (1,0) grid (7,4); 
\draw[ultra thick] 
  (1,0) 
    .. controls (2.5,8.2) and (5.5,-4.2) .. 
  (7,4);  
   \draw[dashed] (1.6,2) -- (1.6,0);
 \draw[dashed] (1.15,0.5) -- (1.15,0);
 \draw[dashed] (2.5,3) -- (2.5,0);
\draw[dashed] (5.5,1) -- (5.5,0);
\draw[dashed] (6.45,2) -- (6.45,0);
\draw [red,thick](1,0)--(1,4);

\node at (0.8,-0.1){$0$};
\node at (1.2,-0.3){$x_1$};
\node at (1.7,-0.3){$x_2$};
\node at (2.5,-0.3){$x_3$};
\node at (4,-0.3){$x_4=\pi$};
\node at (5.5,-0.3){$x_5$};
\node at (6.4,-0.3){$x_6$};
\node at (7,-0.3){$2\pi$};

\node at (0.2,0.5) {$y_1=\pi/4$};
\draw (1,0.5) -- (1.1,0.5);

\node at (0.2,1) {$y_2=\pi/2$};
\node at (0.4,2) {$y_3=\pi$};
\node at (0.1,3) {$y_4=3\pi/2$};
\node at (0.7,4) {$2\pi$};

\draw[->,red,thick] {(1.1,1.7) -- (1.1,2.5)};
\draw[->,thick] {(1.3,1.7) -- (1.65,2.5)};
\draw[<-,thick] {(4.5,1.7) -- (3.65,2.5)};
\draw[->,thick] {(6.2,1.7) -- (6.55,2.5)};

\end{tikzpicture}
\vskip -3.7 truecm
\caption {Representa\c c\~ao gr\'afica da aplica\c c\~ao $f$ na esfera.}\label{aplicercle2}
\end{figure}
\end{center}

\vskip -1 truecm

Os pontos $x_3 = \pi/2$ e $x_5 = 3\pi /2$ s\~ao pontos cr\'iticos onde a derivada se anula. 
O valor $y_3=\pi$ \'e um valor regular tal que $f^{-1}(y_3) = \{ x_2, x_4, x_6 \}$. Nos pontos $x_2$ e $x_6$ a aplica\c c\~ao $df$ preserva a orienta\c c\~ao, mas no ponto $x_4$, a derivada $df$ \'e negativa e 
$df$ inverte a orienta\c c\~ao. Ent\~ao temos:
\begin{equation}\label{critico}
 {\rm grau}(f; y_3) = \sum_{x\in f^{-1}(y_3)} {\rm sign} \; df_x = {\rm sign} \; df_{x_2} + {\rm sign} \; df_{x_4} 
+ {\rm sign} \; df_{x_6} = +1 -1 +1 = +1.
\end{equation}

\begin{center}
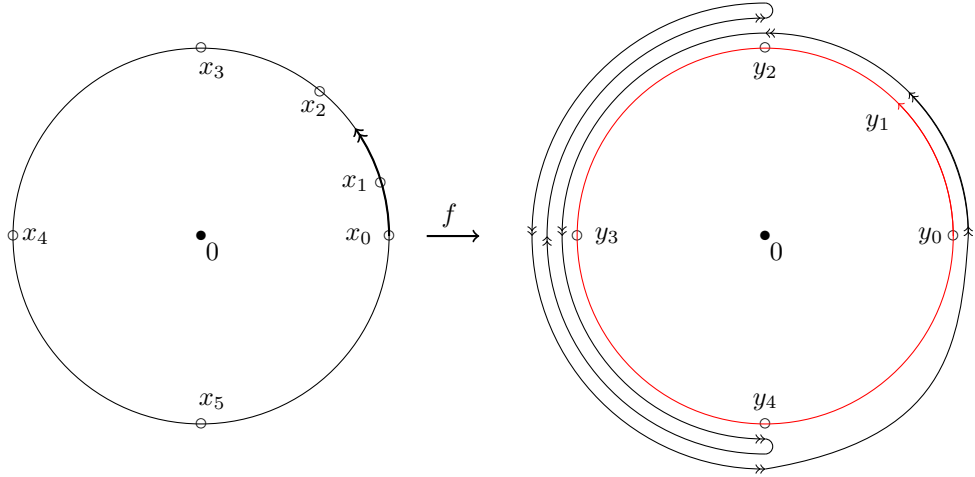
\begin{figure}[H]
\begin{tikzpicture}[baseline]
\node at (0,0) {$\bullet$};
\node at (0.15,-0.2) {$0$};
\node at (2.5cm,0) {$\circ$};
\node at (2.2cm,0) {$y_0$};
\node at (1.5cm,1.5cm) {$y_1$};
\node at (0,2.2cm) {$y_2$};
\node at (0,2.5cm) {$\circ$};
\node at (-2.5cm,0) {$\circ$};
\node at (-2.1cm,0) {$y_3$};
\node at (0,-2.2cm) {$y_4$};
\node at (0,-2.5cm) {$\circ$};

 \draw (-7.5,0) circle (2.5cm);
 \draw [->>,thick] (-5,0) arc (0:35:2.5cm);
\node at (-7.5,0) {$\bullet$};

\node at (-5.4,0) {$x_0$};
\node at (-5,0) {$\circ$};
\node at (-5.45,0.7) {$x_1$};
\node at (-5.115,0.7) {$\circ$};
\node at (-6,1.7) {$x_2$};
\node at (-5.92,1.92) {$\circ$};
\node at (-7.35,-0.2) {$0$};
\node at (-7.5,2.5) {$\circ$};
\node at (-7.35,2.2) {$x_3$};
\node at (-10,0) {$\circ$};
\node at (-9.7,0) {$x_4$};
\node at (-7.5,-2.5) {$\circ$};
\node at (-7.35,-2.2) {$x_5$};

\draw[->,thick] (-4.5,0) -- (-3.8,0);
\node at (-4.2,0) [above] {$f$};

 \draw [red] (0,0) circle (2.5cm);
  \draw [->,red] (2.5cm,0) arc (0:45:2.5cm);
 \draw [->>] (2.7cm,0cm) arc (0:45:2.7cm);
\draw [>>->>] (2.7cm,0cm) arc (0:90:2.7cm);
  \draw [->>] (0cm,2.7cm) arc (90:180:2.7cm);
   \draw [->>] (-2.7cm,0cm) arc (180:270:2.7cm);
 \draw [<<-] (0cm,2.9cm) arc (90:180:2.9cm);
  \draw [<<-](-2.9cm,0cm) arc (180:270:2.9cm);
  \draw  [->>] (0cm,3.1cm) arc (90:180: 3.1cm);
  \draw  [->>] (-3.1cm,0cm) arc (180:270: 3.1cm);
   \draw 
  (2.7,0) 
    .. controls (2.6,-1.5)  and (2.4,-2.7) .. 
  (0,-3.1); 
   \draw (0cm,2.9cm) arc (-90:90:0.1cm);
    \draw (0cm,-2.9cm) arc (-90:90:0.1cm);
\end{tikzpicture}
\medskip
\caption {O grau da aplica\c c\~ao $f$: A ``caminhada'' do ponto $f(x)$ est\'a em preto: 
O ponto  $y_1 = \pi/4$ \'e um valor regular tal que $f^{-1}(y_1) = \{ x_1 \}$. Neste ponto,  
temos a  mesma orienta\c c\~ao do que a orienta\c c\~ao da circunfer\^encia de base (vermelho), 
ent\~ao a aplica\c c\~ao $df$ preserva a orienta\c c\~ao e temos $ {\rm grau}(f;y_1) = +1$.
No ponto $y_3$ 
temos duas vezes a orienta\c c\~ao da circunfer\^encia de base (positiva) e uma vez a orienta\c c\~ao oposta
(negativa). Temos $ {\rm grau}(f;y_3) = +1$ (veja \ref{critico}). 
 }\label{aplicercle1}
\end{figure}
\end{center}
\end{example}

\vskip -1 truecm
Observamos que, por causa da compacidade de $\mathcal C$, ent\~ao $f^{-1}(y)$ consiste 
em um  n\'umero finito de pontos. 

Uma propriedade fundamental  que vamos mostrar \'e que   $ {\rm grau}(f;y) $ n\~ao depende 
do valor regular de $f$, ele \'e denotado por $ {\rm grau}(f)$ e chamado {\it grau de $f$}.

Note que os pontos cr\'iticos n\~ao s\~ao necessariamente pontos isolados. De fato, se um ponto cr\'itico $y_i$ 
n\~ao for isolado, ent\~ao existir\'a um intervalo 
(``horizontal'' no grafo)  $[a,b]$ tal que $f([a,b])=y_i$. \'E f\'acil de ver que podemos substituir 
$f$ por uma pequena perturba\c c\~ao $\widetilde f$ de $f$ (por exemplo uma pequena 
sinusoida) com pontos cr\'iticos isolados e tal que o grau n\~ao depende da perturba\c c\~ao. 

Quando o ponto $x$ faz um turno completo da curva fechada $\mathcal C$, 
o ponto $f(x)$ faz uma caminhada sobre a circunfer\^encia $\Sp^1$. Ele pode fazer as vezes uma volta 
(um returno), as vezes ficar no mesmo lugar
(ponto fixo) mas no fim do turno, ele volta 
no ponto de partida. Chamamos de ponto de volta (ou de returno) os pontos $y_i=f(x_i)$ tais que 
$y$ muda de sentido antes e depois de $y_i$ (veja na Figura \ref{aplicercle1} os pontos $y_2$ e $y_4$).

\begin{lemma}
Os pontos de volta s\~ao   pontos cr\'iticos. 
\end{lemma}
\begin{proof}
Em um ponto de volta, a orienta\c c\~ao da caminhada muda de sentido. Isso significa que o sinal 
da derivada $df_x$ muda.  A derivada \'e nula em um ponto de volta, ent\~ao o ponto \'e cr\'itico.
\end{proof}

\begin{remark}
Um ponto cr\'itico pode n\~ao ser de volta. De fato, um ponto de inflex\~ao do grafo de $f$ 
\'e um ponto cr\'itico mas o sinal de $df_x$ n\~ao muda neste ponto.
\end{remark}

Escolhamos como ponto de partida um ponto $x_0$  de $\mathcal C$ tal que $y_0 = f(x_0)$ 
n\~ao seja ponto cr\'itico (ent\~ao n\~ao seja ponto de volta). 
Quando o ponto $x$ fizer um turno no sentido positivo, o ponto $y=f(x)$ 
poder\'a passar no ponto $y_0$ um certo n\'umero de vezes $p_0$ no sentido positivo 
e um n\'umero de vezes $q_0$ no sentido negativo.  \'E preciso de n\~ao esquecer a passagem inicial,
que  tamb\'em \'e a passagem final e que pode ser positivo ou negativo. 
De fato, $p_0-q_0$ \'e o n\'umero de turnos que faz $y=f(x)$ em $\Sp^1$.
O grau $ {\rm grau}(f;y_0) $ \'e igual a $p_0-q_0$.
 
 \begin{proposition}
O grau n\~ao depende do ponto regular.
 \end{proposition}
 
 \begin{proof}
Seja $y_0$ um ponto regular. Quando $y_0$ muda continuamente, ficando ponto regular, 
\'e claro que o n\'umero $p_0-q_0$ n\~ao muda. Ent\~ao,  o grau \'e constante quando $y$ descreve um intervalo
de pontos regulares na circunfer\^encia $\Sp^1$. Agora, o que acontece quando encontramos um ponto cr\'itico?

$\bullet$ Se o ponto cr\'itico n\~ao for um ponto de volta, o sentido da caminhada n\~ao mudar\'a entre antes 
e depois, ent\~ao o  n\'umero $p_0-q_0$ n\~ao muda e o grau n\~ao muda.

$\bullet$ Se o ponto cr\'itico $y_i= f(x_i)$ for um ponto de volta (por exemplo 
o ponto $y_2$ na Figura \ref{aplicercle1}), ent\~ao a caminhada 
chegar\'a de  um lado de $y_i$  e depois a caminhada volta no mesmo lado. 
Antes do ponto $y_i$ (por exemplo o ponto $y_1$ na Figura \ref{aplicercle1}), 
a caminhada passa $p$ vezes no sentido positivo e $q$ vezes no sentido negativo.
Depois do ponto $y_i$ (por exemplo o ponto $y_3$ na Figura \ref{aplicercle1}),  
a caminhada passa $p+1$ vezes no sentido positivo e $q+1$ 
vezes no sentido negativo. A diferen\c ca $p-q$, isto \'e o grau, n\~ao muda.
\end{proof}

\subsection{Caso de superf\'icies}\label{2.2}

Consideraremos superf\'icies suaves compactas, sem bordas e orientadas, por exemplo, a esfera $\Sp^2$, o toro $\T$, etc. A defini{\c c}\~ao de uma superf\'icie suave \'e a seguinte:

A superf\'icie $\mathcal S$ \'e {\it suave} se existe um recobrimento 
localmente finito 
\footnote{Localmente finito significa que para cada ponto $x$ de $\mathcal S$, 
existe uma vizinhan\c ca $U_x$ de $x$ em $\mathcal S$ que encontra um n\'umero finito de $U_i$.}
$\{ U_i\}$ de $\mathcal{S}$ e, para cada $i$,  existe um homeomorfismo $\varphi_i: U_i \to B_i$,  onde $B_i$ \'e um disco de dimens\~ao dois em $\R^2$, tais que, 
cada vez que $U_i \cap U_j$ n\~ao \'e vazio, a aplica\c c\~ao 
$$h_{ij} : = \varphi_j \circ {\varphi_i}^{-1}: {\varphi_i(U_i \cap U_j)} \to {\varphi_j(U_i \cap U_j)}$$
\'e um difeomorfismo (veja Figura \ref{figura25}). Na seguinte, vamos identificar as abertos $U_i$ 
com bolas $B_i$.

\begin{center}
\begin{figure}[h!]
\begin{tikzpicture}


\draw  (-3cm,4cm) arc (180:30:1);
\draw  (-3cm,4cm) arc (-180:-30:1);
\draw  (-1.5cm, 4.866cm) arc (60:-60:1) arc  (60:-60:-1);
\fill[fill=blue!20] (-1.5cm, 4.866cm) arc (60:-60:1) arc  (60:-60:-1);
\node at (-1.5, 4) {$\bullet$};
\node at (-3,4)[right]{$B_i$};


\draw  (1, 4) circle (1 cm);
\draw  (0.5cm, 4.866cm) arc (60:-60:1) arc (60:-60:-1);
\fill[fill=blue!20] (0.5cm, 4.866cm) arc (60:-60:1) arc (60:-60:-1);
\node at (0.5, 4) {$\bullet$};
\node at (2,4)[left]{$B_j$};

\draw (0,0) ellipse (3cm and 1.5cm);

\draw (-0.5,0) ellipse (1cm and 0.5cm);
\draw (0.5,0) ellipse (1cm and 0.5cm);
\fill[fill=blue!20](0cm, 0.4cm) arc (100:-100:0.4) arc (100:-100: -0.4);
\draw (0cm, 0.4cm) arc (100:-100:0.41) arc (100:-100: -0.41);
\node at (0cm, 0cm){$\bullet$};
\node at (-1cm, 0cm){$U_i$};
\node at (1cm, 0cm){$U_j$};

\draw[->] (-0.9, 4) sin (-0.45, 4.1); 
\draw (-0.45, 4.1) cos  (-0.1,4);
\node at (-0.45, 4.1)[above]{$h_{ij}$};

 \draw[->] (-0.8, 0) arc (180: 200: -5.5);
 \draw (-1.1, 1.8) arc (200: 220: -5);
\node at (-1, 2)[right]{$\varphi_i$};

\draw[->] (0.7, 0) arc (180: 160: 5.5);
 \draw (1, 1.8) arc (160: 180: -5);
\node at (0.7, 2){$\varphi_j$};

\end{tikzpicture}
\caption{Superf\'icie suave.}\label{figura25}
\end{figure}
\end{center}

\pagebreak

\subsubsection{Fibrado tangente a uma superf\'icie suave}\hfill

Da mesma maneira que o fibrado tangente \`a uma curva suave, consideramos para cada ponto $x$ de uma superf\'icie  ${\mathcal{S}}$ o espa{\c c}o vetorial tangente a ${\mathcal{S}}$, de dimens\~ao 2, sobre o corpo $\R$, denotado por $T_x {\mathcal{S}}$. 

Por exemplo, o fibrado tangente \`a esfera $\Sp^2$ \'e $T {\Sp^2} = \bigcup_x T_x {\Sp^2}$. 
Todas as fibras $T_x {\Sp^2}$ s\~ao homeomorfas a $\R^2$. Chamamos $\R^2$ de {\it fibra tipo}. 
O fibrado $T \Sp^2$ \'e localmente trivial, assim como o fibrado tangente 
a qualquer superf\'icie $\mathcal{S}$.  Isso significa que cada ponto $a$ 
da superf\'icie $\mathcal{S}$ admite uma vizinhan{\c c}a $U_a$ tal que 
 $T \mathcal{S} \vert_{U_a}$ \'e isomorfismo a $U_a \times \R^2$:
\begin{equation}\label{trivial}
T \mathcal{S} \vert_{U_a} \cong U_a \times \R^2
\end{equation}
(veja Figura  \ref{figura26}). 
Ent\~ao,  $T\mathcal{S} \vert_{U_a}$ pode ser visto como uma cole{\c c}\~ao de espa{\c c}os vetoriais $\R^2$ sobre $U_a$.

\begin{center}
\begin{figure}[h!]
\begin{tikzpicture}

\draw (0,0) ellipse (4cm and 1 cm);
\node at (-3,-0.2){$\bullet$};
\node at (-1.5,0.2){$\bullet$};
\node at (-1.2,-0.6) {$U_a$};

\draw (-0.6,0.45) -- (-0.6,-0.6) -- (0.6,-0.35) -- (0.6,0.75) -- (-0.6,0.45);
\node at (0,0) {$\bullet$};
\node at (0,0) [below,right] {$a$};
\draw[->] (0,0) --(0.25, 0.5);
\node at (-0.17, 0.05) [above] {$v(a)$};

\draw (1.4,0.3) -- (1.4,-0.75) -- (2.6,-0.45) -- (2.6,0.65) -- (1.4,0.3);
\node at (2,-0.2) {$\bullet$};
\node at (2,-0.2) [below] {$x$};
\draw[->] (2,-0.2) --(1.8, 0.3);
\node at (2.22, 0.05) [above] {$v(x)$};

\node at  (-5,4) {$T \mathcal{S} \vert_{U_a} \cong$};
\node at  (-4.5,3.5) {$U_a \times \R^2$};

\node at (-3,4){$\bullet$};
\draw (-3.5,2) -- (-2.5,3) -- (-2.5,6) -- (-3.5,5) -- (-3.5,2);


\node at (-1.5,4){$\bullet$};
\draw (-2,2) -- (-1,3) -- (-1,6) -- (-2,5) -- (-2,2);

\node at (0,4){$\bullet$};
\draw (-0.5,2) -- (0.5,3) -- (0.5,6) -- (-0.5,5) -- (-0.5,2);
\draw[->] (0,4) --(0.25, 4.5);
\node at (0.2, 4.5) [above] {$v(a)$};

\node at (2,4){$\bullet$};
\draw (1.5,2) -- (2.5,3) -- (2.5,6) -- (1.5,5) -- (1.5,2);
\draw[->] (2,4) --(1.75, 4.5);
\node at (2, 4.5) [above] {$v(x)$};
\node at (1.8, 6.2) {$T_x{\mathcal{S}} \cong \R^2$};

\node at (3.5,0){$\bullet$};
\node at (3.5,4){$\bullet$};
\draw (3,2) -- (4,3) -- (4,6) -- (3,5) -- (3,2);

\draw [->] (5,4) --(6,4);
\node at (5.5,4)[above]{$pr_2$};
\node at (7.5,4){$\bullet$};
\draw (7,2) -- (8,3) -- (8,6) -- (7,5) -- (7,2);
\node at (7.7, 2) {$\R^2$};

\end{tikzpicture}
\caption{O fibrado $T \mathcal{S}$ \'e localmente trivial.}\label{figura26}
\end{figure}
\end{center}

Uma superf\'icie suave orientada \'e uma superf\'icie suave $\mathcal S$ com 
uma escolha de orienta\c c\~ao para cada espa\c co tangente $T_a\mathcal{S}$, tal que para
cada $x\in U_a$, a orienta\c c\~ao de $T_x\mathcal{S}$ corresponde a orienta\c c\~ao de 
$\{x\}  \times \R^2$ pelo isomorfismo (\ref{trivial}). 

O fibrado tangente  $T \Sp^2$ n\~ao \'e globalmente trivial, isto \'e, 
n\~ao h\'a homeomorfismo global  
$$T {\Sp^2} = \Sp^2 \times \R^2.$$
Isso \'e uma das consequ\^encias do Teorema de Poincar\'e-Hopf do qual  veremos a prova (veja Corol\'ario \ref{cheveux}). 

Mas, por exemplo, o fibrado tangente ao toro $\T$ \'e globalmente trivial. 

Do mesmo jeito do que no caso de curva, 
um vetor tangente \`a superf\'icie  suave  $\mathcal{S}$ no ponto $x$ \'e um elemento $v(x)$ de $T_x\mathcal{S}$. 
Podemos definir um vetor tangente como a
 classe de equival\^encia de curvas diferenci\'aveis 
$c : \, \, ]-1,+1[ \to \mathcal S$ desenhadas sobre  a superf\'icie 
$\mathcal{S}$ e tais que $c(0)=x$. 
Temos $v(x) = c'(0)$. 

\subsubsection{Pontos e valores cr\'iticos, grau} \label{grau2}\hfill

Seja $f: \mathcal{S}  \to \Sp^2$ 
uma aplica\c c\~ao suave   de uma superf\'icie suave $\mathcal{S}$ em $\Sp^2$, 
do mesmo jeito do que no caso duma curva, 
 a {\it derivada}  da aplica\c c\~ao suave 
$f: \mathcal{S}  \to \Sp^2$,  em um ponto $x$ da curva $\mathcal{S}$, chamada de 
{\it aplica\c c\~ao linear tangente} e denotada por
\begin{equation} \label{derivee2}
df_x : T_x\mathcal{S} \to T_{f(x)}\Sp^2, \qquad \text{ ou seja } \qquad df_x :{\mathbb R}^2 \to {\mathbb R}^2 
\end{equation}
\'e definida por 
\begin{equation} \label{derivee3}
df_x(v) = df_x(c'(0)) = (f \circ c)'(0)
\end{equation}
para toda curva $c : \, \, ]-1,+1[ \to \mathcal S$ tal que $c(0)=x$ e $v(x) = c'(0)$. 
De fato, $f\circ c$ \'e uma curva sobre  $\Sp^2$ 
tal que $(f\circ c)(0) = f(x)$.  Neste caso,  $(f \circ c)'(0)$ \'e um vetor tangente 
\`a  $\Sp^2$ no ponto $f(x)$. 

Seja $f: \mathcal{S}  \to \Sp^2$ 
uma aplica\c c\~ao suave   de uma superf\'icie suave $\mathcal{S}$ em $\Sp^2$, 
um ponto $a\in \mathcal{S}$ tal que a aplica\c c\~ao derivada 
$$df_a : T_a\mathcal{S} \to T_{f(a)}\Sp^2$$
tem o posto menor do que 2 \'e chamado de  {\it ponto cr\'itico} de $f$. 

Um ponto $a$ da superf\'icie $\mathcal{S}$ \'e um ponto regular da aplica\c c\~ao $f$
se ele n\~ao \'e cr\'itico. As imagens de pontos cr\'iticos s\~ao chamadas de {\it valores cr\'iticos} 
e os demais pontos  s\~ao chamados de {\it valores regulares} da aplica\c c\~ao.

Como no caso de dimens\~ao 1, um ponto $y\in \Sp^2$ \'e um valor regular  de $f$ se $f^{-1}(y)=\emptyset$
ou todos  os pontos $x\in f^{-1}(y)$ 
s\~ao pontos regulares de $f$. Em um valor regular $y$ de $f$, podemos considerar o n\'umero inteiro 
$$ {\rm grau}(f;y) = \sum_{x\in f^{-1}(y)} {\rm sign} \; df_x.$$
Assim, podemos ver que o  $ {\rm grau}(f;y) $ n\~ao depende do valor regular de $f$. Portanto, podemos denot\'a-lo simplesmente por $ {\rm grau}(f)$ e chamaremos de {\it grau de $f$}.

\section{Singularidades de campos de vetores  tangentes --  \'Indice em um ponto singular}
\label{indices1}
\subsection{Singularidades de campos de vetores  tangentes a uma superf\'icie suave}\label{indice1}
Sejam $\mathcal{S}$ uma superf\'icie suave e $U_a$ uma vizinhan\c ca de um ponto $a\in \Su$. 
Um campo de vetores sobre $U_a$, tangente a $\mathcal{S}$ \'e uma aplica{\c c}\~ao cont\'inua:
$$v: U_a \to T \mathcal{S} \vert_{U_a}$$
tal que, para cada $x \in U_a$, $v(x) \in T_x \mathcal{S} \cong \R^2$. 

Uma {\it singularidade} de um campo de vetores $v$ \'e um ponto $a$ tal que $v(a) = 0$. 

A singularidade $a$ \'e isolada se existir uma vizinhan{\c c}a $V_a$ 
de $a$ tal que para cada 
ponto $x$ de $V_a$ diferente de $a$, teremos $v(x) \neq 0$. 

Mostraremos que, a uma singularidade isolada de $v$, podemos associar um 
num\'ero inteiro que chamaremos de {\it \'indice do campo de vetores $v$ no ponto $a$} e denotaremos 
por $I(v, a)$. O \'indice de um campo de vetores pode ser definido  de muitas maneiras equivalentes 
em  v\'arios contextos: usando grupos de homologia, usando integrais de formas diferenciais, etc.
Aqui, vamos usar  a defini\c c\~ao em termos de aplica\c c\~ao de Gauss e de grau (veja tamb\'em 
\cite{Li1,Li2}). 

Suponhando a superf\'icie $\mathcal{S}$ orientada, consideramos  uma vizinhan{\c c}a $V_a$ 
de $a$ homeomorfa a um aberto de $\R^2$ e tal que $a$ seja a \'unica singularidade de $v$ dentro de 
$V_a$.  Isto \'e, $a$ \'e uma singularidade isolada de $v$ (veja Figura \ref{figura27}). 
Podemos considerar um segundo exemplar 
de $\R^2$, denotado por  $\widetilde{\R}^2$, com a mesma orienta{\c c}\~ao.

\begin{center}
\begin{figure}[h!]
\begin{tikzpicture}
\draw (0,0) circle (1.5 cm);
\node at (0,0)[below]{a};
\node at (0,0) {$\bullet$};
\node at (0, -1.5)[above]{$B_a$};
\node at (1.5, -1.5)[above]{${\bf S}_a^1$};
\node at (1,1.118) {$\bullet$};
\node at (1,1.118)[below] {$x$};
\draw[->] (1,1.118) --(1.7, 1);
\node at (0.5,1.4142) {$\bullet$};
\node at (0.5,1.4142)[below] {$y$};
\draw[->] (0.5,1.4142) --(1.2, 1.6);
\node at ((-0.5, 1.4142) {$\bullet$};
\node at ((0, 1.5) {$\bullet$};
\draw[->] (0, 1.5) --(0.35, 2.12);
\draw[->] (-0.5, 1.4142) --(- 1, 2.0);
\node at (-0.5, 1.4142)[below] {$z$};
\node at (-1 , 1.118) {$\bullet$};
\draw[->] (-1 , 1.118) --(-1.7, 0.9322);
\draw[->, very thick] (0.5,0.81018) arc (60:90:1.4);

\draw[->] (3,0) -- (4,0);
\node at (3.5,0) [above] {$\gamma_v$};


\draw (7,0) circle (1.5 cm);
\node at (7,0)[below]{0};
\node at (7,0) {$\bullet$};
\node at (8.5, -1.5)[above]{${\bf S}^1$};
\draw[->, very thick] (7,0) --(8.5, -0.25);
\node at (8.5, -0.25)[right]{$\gamma_v(x)$};
\node at (7.75, -0.32) {$\widetilde v(x)$};
\draw[->, very thick] (7,0) --(8.5, 0.35);
\node at (8.5, 0.35)[right]{$\gamma_v(y)$};
\node at (7.75, 0.4) {$\widetilde v(y)$};

\draw[->, very thick] (7,0) --(6, 1.15);
\node at (6.1, 1.15)[left]{$\gamma_v(z)$};
\node at (6.12, 0.4) {$\widetilde v(z)$};

\node at (7,-2){$\widetilde{\R}^2$};

\draw[->, very thick] (8, 1.5) arc (55:90:1.6);

\draw (2.5, 0) arc (0:45:2.5) arc (45:60:1) arc (60:90:2.5) arc (90: 120:4) arc (120: 150: 2.5) 
arc (150: 180: 1.9); 
\draw (-2.85,-0.08) sin (0, -2.5);
\draw (-0,-2.5) parabola[parabola height=-0.7cm] +(2.5, 2.5);

\node at (0, -2.3)[above] {$V_a$};
\node at (2,-2){$\R^2$};

\end{tikzpicture}
\caption{A aplica\c c\~ao de Gauss.}\label{figura27}
\end{figure}
\end{center}


Seja ${B}_a$ uma pequena bola de centro $a$ dentro de $V_a$. Ao longo da esfera $\Sp^1_a$, que \'e a borda de ${B}_a$, o campo de vetores $v$ n\~ao tem singularidades.

Para cada ponto $x$ de $\Sp^1_a$, consideramos o vetor 
$$\widetilde v(x) = \frac{v(x)}{\Vert v(x) \Vert}$$ 
em
$\widetilde{\R}^2$, 
onde $\Vert v(x) \Vert$ \'e a norma euclidiana de $v(x)$. 

O vetor $\widetilde v((x)$ tem tamanho 1 e a sua extremidade, $\gamma_v(x)$, pertence \`a esfera ${\Sp}^1$ de raio 1 e  centro na origem de $\widetilde{\R}^2$. 

A aplica{\c c}\~ao $\gamma_v: \Sp^1_a \to {\Sp}^1$  \'e chamada de {\it aplica{\c c}\~ao de Gauss}. 

Quando o ponto $x$ fizer um turno, no sentido positivo de $\Sp^1_a$, o ponto $\gamma_v (x)$  far\'a certos ``caminhos'' (peda\c cos de turno) da esfera ${\Sp}^1$ no sentido positivo ou negativo. 

\begin{remark}\label{orientation}
A orienta\c c\~ao positiva ou negativa do peda\c co de caminho do ponto $\gamma_v (x)$ 
providencia respetivamente a orienta\c c\~ao positiva ou negativa de $\widetilde{\R}^2$. 
\end{remark} 

Seja $\alpha$ um ponto de ${\Sp}^1$, um caminho passando no ponto $\alpha$ no sentido positivo vale $+1$ e um caminho passando no ponto $\alpha$ no sentido negativo vale $-1$. 

A soma destes valores \'e,
por defini\c c\~ao,  o \'indice $I(v;\alpha)$ do campo de vetores $v$ no ponto $a$. 
\'E f\'acil de ver que o  \'indice $I(v,a)$ \'e igual ao grau de $\gamma_v$ (veja \S \ref{grau}).

\begin{remark}\label{indice2}
Observamos que a aplica\c c\~ao de Gauss $\widetilde\gamma_v$ \'e a composta 
\begin{eqnarray*}
\Sp^1_a   \longrightarrow & T \mathcal{S} \vert_{\Sp^1_a} \cong \Sp^1_a \times \R^2 
& \buildrel{pr_2}\over \longrightarrow  \quad \R^2 \\
 x \quad   \mapsto &  v(x)  \quad & \mapsto \quad \widetilde\gamma_v(x) 
 \end{eqnarray*}
(veja Defini\c c\~ao \ref{naosei}). Esta propriedade \'e a base da teoria de obstru\c c\~ao, 
uma das teorias que permite de introduzir a teoria de classes caracter\'isticas.
\end{remark} 

\begin{example}
Consideramos o campo de vetores $v$ na Figura \ref{aplicercle3}. A aplica\c c\~ao de Gauss 
\'e a aplica\c c\~ao que consideramos no exemplo \ref{exemplo1} e temos $I(v,0) = +1$.

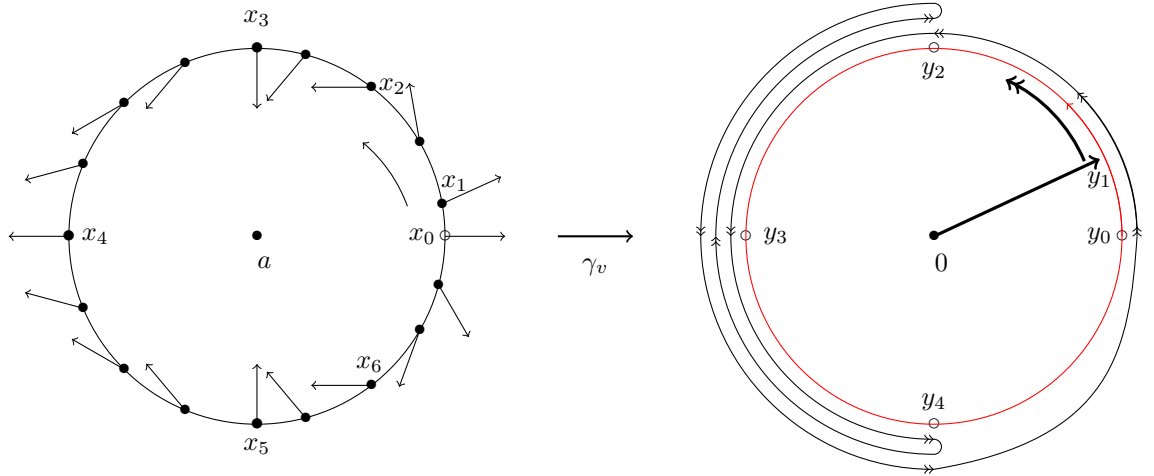
\begin{figure}[H]
\begin{tikzpicture}[baseline]
\node at (0,0) {$\bullet$};
\node at (2.5cm,0) {$\circ$};
\node at (2.2cm,0) {$x_0$};
\node at (0,2.5cm) {$\circ$};
\node at (-2.5cm,0) {$\circ$};
\node at (0,-2.5cm) {$\circ$};

\node at (0.1,-0.35) {$a$};
\node at (2.63,0.7) {$x_1$};
\node at (1.8,2) {$x_2$};
\node at (0,2.9cm) {$x_3$};
\node at (-2.15,0) {$x_4$};
\node at (0,-2.8) {$x_5$};
\node at (1.5,-1.7) {$x_6$};

 \draw (0,0) circle (2.5);
 
   \draw [->] (2cm,0.4cm) arc (20:50:2cm);

\path (1*10:2.5) coordinate (A1);
\node at (A1) {$\bullet$};
 \path (2*15:2.5) coordinate (A2);
 \node at (A2) {$\bullet$};
  \path (7*7.5:2.5) coordinate (A3bis);
 \node at (A3bis) {$\bullet$};
  \path (5*15:2.5) coordinate (A5);
 \node at (A5) {$\bullet$};
   \path (6*15:2.5) coordinate (A6);
 \node at (A6) {$\bullet$};
 
  \draw[->] (2.5,0) -- (3.3,0); 
  \begin{scope}[shift={(A1)},rotate=60]; 
  \draw[-> ] (A1) -- (0.7,-0.5);
  \end{scope}
   \begin{scope}[shift={(A2)},rotate=100];
    \draw[-> ] (A2) -- (0.8,0);
  \end{scope}
   \begin{scope}[shift={(A3bis)},rotate=180]; 
    \draw[-> ] (A3bis) -- (0.8,0);
  \end{scope}
 \begin{scope}[shift={(A5)},rotate=230];
    \draw[-> ] (A5) -- (0.8,0);
  \end{scope}
  \begin{scope}[shift={(A6)},rotate=270];
    \draw[-> ] (A6) -- (0.8,0);
  \end{scope}


 \path (5*22.5:2.5) coordinate (A7);
 \node at (A7) {$\bullet$};
 \path (6*22.5:2.5) coordinate (A8);
 \node at (A8) {$\bullet$};
  \path (7*22.5:2.5) coordinate (A9);
 \node at (A9) {$\bullet$};
  \path (8*22.5:2.5) coordinate (A10);
 \node at (A10) {$\bullet$};
  \path (9*22.5:2.5) coordinate (A11);
 \node at (A11) {$\bullet$};
 \path (10*22.5:2.5) coordinate (A12);
 \node at (A12) {$\bullet$};
  \path (11*22.5:2.5) coordinate (A13);
 \node at (A13) {$\bullet$};
  \path (12*22.5:2.5) coordinate (A14);
 \node at (A14) {$\bullet$};
 
 \begin{scope}[shift={(A7)},rotate=230];
  \draw[-> ] (A7) -- (0.80,0);
  \end{scope}
   \begin{scope}[shift={(A8)},rotate=210];
    \draw[-> ] (A8) -- (0.8,0);
  \end{scope}
  \begin{scope}[shift={(A9)},rotate=195];
    \draw[-> ] (A9) -- (0.8,0);
  \end{scope}
 \begin{scope}[shift={(A10)},rotate=180];
    \draw[-> ] (A10) -- (0.8,0);
  \end{scope}
 \begin{scope}[shift={(A11)},rotate=165];
    \draw[-> ] (A11) -- (0.8,0);
  \end{scope}
  \begin{scope}[shift={(A12)},rotate=150];
    \draw[-> ] (A12) -- (0.8,0);
  \end{scope}
  \begin{scope}[shift={(A13)},rotate=130];
  \draw[-> ] (A13) -- (0.80,0);
  \end{scope}
   \begin{scope}[shift={(A14)},rotate=90];
    \draw[-> ] (A14) -- (0.8,0);
  \end{scope}


\path (19*15:2.5) coordinate (A15);
\node at (A15) {$\bullet$};
\path (41*7.5:2.5) coordinate (A16);
\node at (A16) {$\bullet$};
  \path (22*15:2.5) coordinate (A18);
 \node at (A18) {$\bullet$};
  \path (23*15:2.5) coordinate (A19);
 \node at (A19) {$\bullet$};
  
\begin{scope}[shift={(A15)},rotate=130];
    \draw[-> ] (A15) -- (0.8,0);
  \end{scope}
\begin{scope}[shift={(A16)},rotate=180];
   \draw[-> ] (A16) -- (0.8,0);
 \end{scope}
   \begin{scope}[shift={(A18)},rotate=250];
    \draw[-> ] (A18) -- (0.8,0);
  \end{scope}
  \begin{scope}[shift={(A19)},rotate=300];
    \draw[-> ] (A19) -- (0.8,0);
  \end{scope}

 \draw[->, thick ] (4,0) -- (5,0);
 \node at (4.5,-0.4) {$\gamma_v$};
 
\node at (9.1,-0.35) {$0$};
\node at (9,0) {$\bullet$};
\node at (11.5cm,0) {$\circ$};
\node at (11.2cm,0) {$y_0$};
\node at (11.2cm,0.75cm) {$y_1$};
\node at (9,2.2cm) {$y_2$};
\node at (9,2.5cm) {$\circ$};
\node at (6.5cm,0) {$\circ$};
\node at (6.9cm,0) {$y_3$};
\node at (9,-2.2cm) {$y_4$};
\node at (9,-2.5cm) {$\circ$};

 \draw [red] (9,0) circle (2.5cm);
  \draw [->,red] (11.5cm,0) arc (0:45:2.5cm);
 \draw [->>] (11.7cm,0cm) arc (0:45:2.7cm);
\draw [>>->>] (11.7cm,0cm) arc (0:90:2.7cm);
  \draw [->>] (9cm,2.7cm) arc (90:180:2.7cm);
   \draw [->>] (6.3cm,0cm) arc (180:270:2.7cm);
 \draw [<<-] (9cm,2.9cm) arc (90:180:2.9cm);
  \draw [<<-](6.1cm,0cm) arc (180:270:2.9cm);
  \draw  [->>] (9cm,3.1cm) arc (90:180: 3.1cm);
  \draw  [->>] (5.9cm,0cm) arc (180:270: 3.1cm);
   \draw 
  (11.7,0) 
    .. controls (11.6,-1.5) and (11.4,-2.7) .. 
  (9,-3.1); 

   \draw[->>,very thick] (11,1) arc (23:65:2.1);
   \draw (9cm,2.9cm) arc (-90:90:0.1cm);
    \draw (9cm,-2.9cm) arc (-90:90:0.1cm);
    
    \begin{scope}[shift={(9,0)},rotate=25];
    \draw[-> ,very thick] (0,0) -- (2.43,0);
  \end{scope}
  
\end{tikzpicture}
\medskip
\caption {Aplica\c c\~ao de Gauss. Aqui: $y_0 = \gamma_v (x_0)$, $y_1 = \gamma_v (x_1) $, $y_2 = \gamma_v (x_5) $, 
$y_3 = \gamma_v (x_2) = \gamma_v (x_4) =\gamma_v (x_6) $
e  $y_4 = \gamma_v (x_3) $. 
}\label{aplicercle3}
\end{figure}
\end{example}

Aqui temos alguns exemplos de \'indices:
\vglue - 0.5truecm

\begin{center}
\begin{figure}[h!]
\begin{tikzpicture} 


\draw (0,0) circle (1.5cm);

\node at (0,0) {$\bullet$}; 
\node at (0,0)[below] {$a$}; 

\draw[->] (0.5, 1.414) -- (1, 2.5);
\draw[->] (-0.5, 1.414) -- (-1, 2.5);
\draw[->] (0.5, -1.414) -- (1, -2.5);
\draw[->] (-0.5, -1.414) -- (-1, -2.5);

\draw[->] (1, 1.118) -- (2, 2);
\draw[->] (-1, 1.118) -- (-2, 2);
\draw[->] (1, -1.118) -- (2, -2);
\draw[->] (-1, -1.118) -- (-2, -2);

\draw[->] (1.3, 0.748) -- (2.5, 1);
\draw[->] (-1.3, 0.748) -- (-2.5, 1);
\draw[->] (1.3, -0.748) -- (2.5, -1);
\draw[->] (-1.3, -0.748) -- (-2.5, -1);

\draw[->] (1.5, 0) -- (2.7, 0);
\draw[->] (-1.5, 0) -- (-2.7, 0);
\draw[->] (0, 1.5) -- (0, 2.7);
\draw[->] (0, -1.5) -- (0, -2.7);


\draw (7,0) circle (1.5cm);

\node at (7,0) {$\bullet$}; 
\node at (7,0)[below] {$a$}; 

\draw[->] (7.5, 1.414) -- (7.2, 0.7);
\draw[->] (6.5, 1.414) -- (6.7, 0.7);
\draw[->] (7.5, -1.414) -- (7.2, -0.7);
\draw[->] (6.5, -1.414) -- (6.7, -0.7);

\draw[->] (8, 1.118) -- (7.5, 0.5);
\draw[->] (6, 1.118) -- (6.5, 0.5);
\draw[->] (8, -1.118) -- (7.5, -0.5);
\draw[->] (6, -1.118) -- (6.5, -0.5);

\draw[->] (8.3, 0.748) -- (7.6, 0.3);
\draw[->] (8.3, -0.748) -- (7.6, -0.3);
\draw[->] (5.7, 0.748) -- (6.2, 0.3);
\draw[->] (5.7, -0.748) -- (6.2, -0.3);

\draw[->] (7, 1.5) -- (7, 0.7);
\draw[->] (7, -1.5) -- (7, -0.7);
\draw[->] (8.5, 0) -- (7.7, 0);
\draw[->] (5.5, 0) -- (6.2, 0);

\end{tikzpicture} 
\caption{Campo radial saindo e campo radial entrando: $I(v,a) = +1$.\hfill\break
\label{figura28}}
\end{figure}
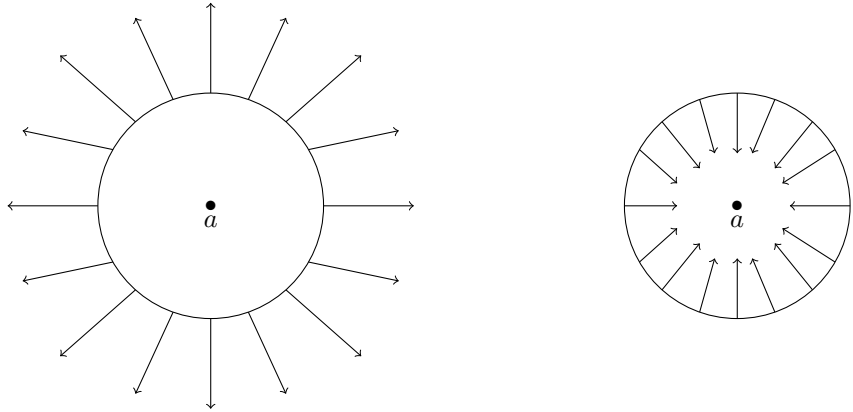
\end{center}

\begin{center}
\begin{figure}[h!]
\begin{tikzpicture} 


\draw (0,0) circle (1.5cm);

\node at (0,0) {$\bullet$}; 
\node at (0,0)[below] {$a$}; 

\draw[->] (0.5, 1.414) -- (1.5, 1.414);
\draw[->] (-0.5, 1.414) -- (0.2, 1.414);
\draw[->] (0.5, -1.414) -- (1.5, -1.414);
\draw[->] (-0.5, -1.414) -- (0.2, -1.414);

\draw[->] (1, 1.118) -- (2, 1.118);
\draw[->] (-1, 1.118) -- (0, 1.118);
\draw[->] (1, -1.118) -- (2, -1.118);
\draw[->] (-1, -1.118) -- (0, -1.118);

\draw[->] (1.3, 0.748) -- (2.3, 0.748);
\draw[->] (-1.3, 0.748) -- (-0.3, 0.748);
\draw[->] (1.3, -0.748) -- (2.3, -0.748);
\draw[->] (-1.3, -0.748) -- (-0.3, -0.748);

\draw[->] (1.5, 0) -- (2.3, 0);
\draw[->] (-1.5, 0) -- (-0.5, 0);
\draw[->] (0, 1.5) -- (1, 1.5);
\draw[->] (0, -1.5) -- (0.8, -1.5);

\end{tikzpicture} 
\caption{Campo de \'indice nulo.}\label{figura30}
\end{figure}
\end{center}


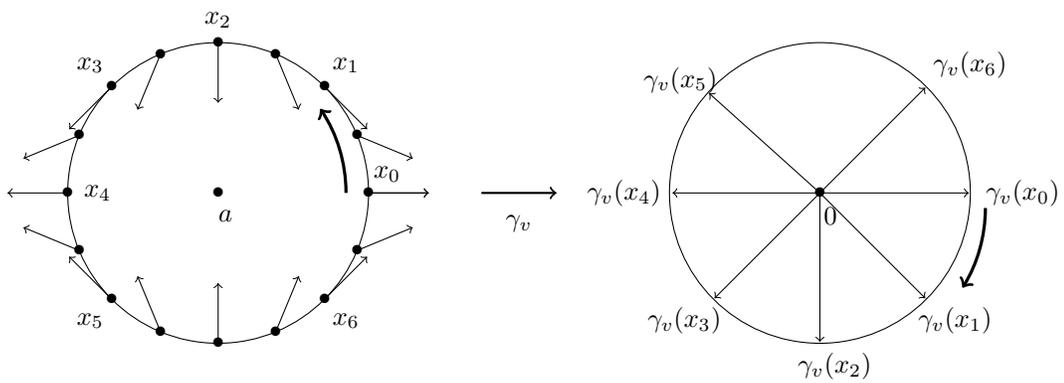
\begin{figure}[H]
\begin{center}
\begin{tikzpicture}[baseline]
\node at (0,0) {$\bullet$};
 \node at (0.1,-0.3) {$a$};
\node at (8,0) {$\bullet$};
 \node at (8.14,-0.3) {$0$};

 \draw (0,0) circle (2);
 \draw (8,0) circle (2);

 \draw[->,  thick ] (3.5,0) -- (4.5,0);
 \node at (4,-0.4) {$\gamma_v$};
 
 
   \draw [->, very thick] (1.7cm,0cm) arc (0:34:2cm);
\draw [->, very thick] (10.2cm,-0.2cm) arc (0:-32:2cm);
 
 \path (0*22.5:2) coordinate (A0);
\node at (A0) {$\bullet$};
\path (1*22.5:2) coordinate (A1);
\node at (A1) {$\bullet$};
 \path (2*22.5:2) coordinate (A2);
 \node at (A2) {$\bullet$};
 \path (3*22.5:2) coordinate (A3);
 \node at (A3) {$\bullet$};
  \path (4*22.5:2) coordinate (A4);
 \node at (A4) {$\bullet$};
   \path (5*22.5:2) coordinate (A5);
 \node at (A5) {$\bullet$};
 \path (6*22.5:2) coordinate (A6);
\node at (A6) {$\bullet$};
\path (7*22.5:2) coordinate (A7);
\node at (A7) {$\bullet$};
 \path (8*22.5:2) coordinate (A8);
 \node at (A8) {$\bullet$};
 \path (9*22.5:2) coordinate (A9);
 \node at (A9) {$\bullet$};
  \path (10*22.5:2) coordinate (A10);
 \node at (A10) {$\bullet$};
   \path (11*22.5:2) coordinate (A11);
 \node at (A11) {$\bullet$};
\path (12*22.5:2) coordinate (A12);
 \node at (A12) {$\bullet$};
\path (13*22.5:2) coordinate (A13);
 \node at (A13) {$\bullet$};
   \path (14*22.5:2) coordinate (A14);
 \node at (A14) {$\bullet$};
\path (15*22.5:2) coordinate (A15);
 \node at (A15) {$\bullet$};

\node at (2.25,0.25) {$x_0$};
\node at (1.7,1.7) {$x_1$};
\node at (0,2.3) {$x_2$};
\node at (-1.7,1.7) {$x_3$};
\node at (-1.6,0) {$x_4$};
\node at (-1.7,-1.7) {$x_5$};
\node at (1.7,-1.7) {$x_6$};

\node at (10.7,0) {$\gamma_v(x_0)$};
\node at (9.8,-1.7) {$\gamma_v(x_1)$};
\node at (8.2,-2.3) {$\gamma_v(x_2)$};
\node at (6.2,-1.7) {$\gamma_v(x_3)$};
\node at (5.4,0) {$\gamma_v(x_4)$};
\node at (6.15,1.5) {$\gamma_v(x_5)$};
\node at (10,1.7) {$\gamma_v(x_6)$};
 
  \draw[->] (2,0) -- (2.8,0); 
  \begin{scope}[shift={(A1)},rotate=337.5];
  \draw[-> ] (A1) -- (0.8,0);
  \end{scope}
   \begin{scope}[shift={(A2)},rotate=315];
    \draw[-> ] (A2) -- (0.8,0);
  \end{scope}
  \begin{scope}[shift={(A3)},rotate=292.5];
    \draw[-> ] (A3) -- (0.8,0);
  \end{scope}
 \begin{scope}[shift={(A4)},rotate=270];
    \draw[-> ] (A4) -- (0.8,0);
  \end{scope}
 \draw[->] (2,0) -- (2.8,0); 
  \begin{scope}[shift={(A5)},rotate=247.5];
  \draw[-> ] (A5) -- (0.8,0);
  \end{scope}
   \begin{scope}[shift={(A6)},rotate=225];
    \draw[-> ] (A6) -- (0.8,0);
  \end{scope}
  \begin{scope}[shift={(A7)},rotate=202.5];
    \draw[-> ] (A7) -- (0.8,0);
  \end{scope}
 \begin{scope}[shift={(A8)},rotate=180];
    \draw[-> ] (A8) -- (0.8,0);
  \end{scope}
  \begin{scope}[shift={(A9)},rotate=157.5];
    \draw[-> ] (A9) -- (0.8,0);
  \end{scope}
  \begin{scope}[shift={(A10)},rotate=135];
    \draw[-> ] (A10) -- (0.8,0);
  \end{scope}
 \begin{scope}[shift={(A11)},rotate=112.5];
    \draw[-> ] (A11) -- (0.8,0);
  \end{scope}
\begin{scope}[shift={(A12)},rotate=90];
    \draw[-> ] (A12) -- (0.8,0);
  \end{scope}
  \begin{scope}[shift={(A13)},rotate=67.5];
    \draw[-> ] (A13) -- (0.8,0);
  \end{scope}
 \begin{scope}[shift={(A14)},rotate=45];
    \draw[-> ] (A14) -- (0.8,0);
  \end{scope}
\begin{scope}[shift={(A15)},rotate=22.5];
    \draw[-> ] (A15) -- (0.8,0);
  \end{scope}

\draw[<-] (6.04,0) -- (8,0);
\begin{scope}[shift={(8,0)},rotate=45];
    \draw[-> ] (0,0) -- (1.98,0);
  \end{scope}
  \begin{scope}[shift={(8,0)},rotate=138];
    \draw[-> ] (0,0) -- (1.98,0);
  \end{scope}
\begin{scope}[shift={(8,0)},rotate=225];
    \draw[-> ] (0,0) -- (1.98,0);
  \end{scope}
  \begin{scope}[shift={(8,0)},rotate=270];
    \draw[-> ] (0,0) -- (1.98,0);
  \end{scope}
   \begin{scope}[shift={(8,0)},rotate=315];
    \draw[-> ] (0,0) -- (1.98,0);
  \end{scope}
    \begin{scope}[shift={(8,0)},rotate=0];
    \draw[-> ] (0,0) -- (1.98,0);
  \end{scope}
  
\end{tikzpicture}
\end{center}

\caption{Campo de \'indice $I(v,a)=-1$.}\label{figura34}
\end{figure}

\begin{figure}[H]
\begin{center}
\begin{tikzpicture}[baseline]
\node at (0,0) {$\bullet$};
 \node at (0.1,-0.3) {$a$};
\node at (8,0) {$\bullet$};
 \node at (8.02,-0.3) {$0$};
 
 \draw (0,0) circle (2);
 \draw (8,0) circle (2);

 \draw[->,  thick ] (3.5,0) -- (4.5,0);
 \node at (4,-0.4) {$\gamma_v$};
 
 
   \draw [->, very thick] (1.7cm,0cm) arc (0:30:2cm);
\draw [->, very thick] (9.7cm,0.1cm) arc (0:30:2cm);
 
 \path (0*22.5:2) coordinate (A0);
\node at (A0) {$\bullet$};
\path (1*22.5:2) coordinate (A1);
\node at (A1) {$\bullet$};
 \path (2*22.5:2) coordinate (A2);
 \node at (A2) {$\bullet$};
 \path (3*22.5:2) coordinate (A3);
 \node at (A3) {$\bullet$};
  \path (4*22.5:2) coordinate (A4);
 \node at (A4) {$\bullet$};
   \path (5*22.5:2) coordinate (A5);
 \node at (A5) {$\bullet$};
 \path (6*22.5:2) coordinate (A6);
\node at (A6) {$\bullet$};
\path (7*22.5:2) coordinate (A7);
\node at (A7) {$\bullet$};
 \path (8*22.5:2) coordinate (A8);
 \node at (A8) {$\bullet$};
 \path (9*22.5:2) coordinate (A9);
 \node at (A9) {$\bullet$};
  \path (10*22.5:2) coordinate (A10);
 \node at (A10) {$\bullet$};
   \path (11*22.5:2) coordinate (A11);
 \node at (A11) {$\bullet$};
\path (12*22.5:2) coordinate (A12);
 \node at (A12) {$\bullet$};
\path (13*22.5:2) coordinate (A13);
 \node at (A13) {$\bullet$};
   \path (14*22.5:2) coordinate (A14);
 \node at (A14) {$\bullet$};
\path (15*22.5:2) coordinate (A15);
 \node at (A15) {$\bullet$};

\node at (2.25,0.25) {$x_0$};
\node at (2.25,0.8) {$x_1$};
\node at (0.7,1.5) {$x_2$};
\node at (-0.7,1.5) {$x_3$};
\node at (-2.2,0.8) {$x_4$};
\node at (-2.4,0) {$x_5$};
\node at (-2.2,-0.8) {$x_6$};
\node at (-0.7,-1.5)  {$x_7$};
\node at (0.7,-1.5) {$x_8$};
\node at (1.5,-0.8) {$x_9$};

\node at (10.8,0) {$\gamma_v(x_0) =$};
\node at (11.2,-0.5) {$\gamma_v(x_5)$};
\node at (10,1.7) {$\gamma_v(x_1)=$};
\node at (10.4,1.2) {$\gamma_v(x_6) $};
\node at (5.7,2.1) {$\gamma_v(x_2)=$};
\node at (6.1,1.6) {$\gamma_v(x_7)$};
\node at (5.8,-1.4) {$\gamma_v(x_3)=$};
\node at (6.2,-1.9) {$\gamma_v(x_8)$};
\node at (9.9,-1.7) {$\gamma_v(x_4)=$};
\node at (10.3,-2.2) {$\gamma_v(x_9)$};

  \draw[->] (2,0) -- (2.8,0); 
  \begin{scope}[shift={(A1)},rotate=45];
  \draw[-> ] (A1) -- (0.8,0);
  \end{scope}
   \begin{scope}[shift={(A2)},rotate=90];
    \draw[-> ] (A2) -- (0.8,0);
  \end{scope}
  \begin{scope}[shift={(A3)},rotate=135];
    \draw[-> ] (A3) -- (0.8,0);
  \end{scope}
 \begin{scope}[shift={(A4)},rotate=180];
    \draw[-> ] (A4) -- (0.8,0);
  \end{scope}
 \draw[->] (2,0) -- (2.8,0); 
  \begin{scope}[shift={(A5)},rotate=225];
  \draw[-> ] (A5) -- (0.8,0);
  \end{scope}
   \begin{scope}[shift={(A6)},rotate=270];
    \draw[-> ] (A6) -- (0.8,0);
  \end{scope}
  \begin{scope}[shift={(A7)},rotate=315];
    \draw[-> ] (A7) -- (0.8,0);
  \end{scope}
 \begin{scope}[shift={(A8)},rotate=360];
    \draw[-> ] (A8) -- (0.8,0);
  \end{scope}
  \begin{scope}[shift={(A9)},rotate=45];
    \draw[-> ] (A9) -- (0.8,0);
  \end{scope}
  \begin{scope}[shift={(A10)},rotate=90];
    \draw[-> ] (A10) -- (0.8,0);
  \end{scope}
 \begin{scope}[shift={(A11)},rotate=135];
    \draw[-> ] (A11) -- (0.8,0);
  \end{scope}
\begin{scope}[shift={(A12)},rotate=180];
    \draw[-> ] (A12) -- (0.8,0);
  \end{scope}
  \begin{scope}[shift={(A13)},rotate=225];
    \draw[-> ] (A13) -- (0.8,0);
  \end{scope}
 \begin{scope}[shift={(A14)},rotate=270];
    \draw[-> ] (A14) -- (0.8,0);
  \end{scope}
\begin{scope}[shift={(A15)},rotate=315];
    \draw[-> ] (A15) -- (0.8,0);
  \end{scope}

\begin{scope}[shift={(8,0)},rotate=45];
    \draw[-> ] (0,0) -- (1.98,0);
  \end{scope}
  \begin{scope}[shift={(8,0)},rotate=135];
    \draw[-> ] (0,0) -- (1.98,0);
  \end{scope}
\begin{scope}[shift={(8,0)},rotate=225];
    \draw[-> ] (0,0) -- (1.98,0);
  \end{scope}
    \begin{scope}[shift={(8,0)},rotate=315];
    \draw[-> ] (0,0) -- (1.98,0);
  \end{scope}
    \begin{scope}[shift={(8,0)},rotate=0];
    \draw[-> ] (0,0) -- (1.98,0);
  \end{scope}
  
\end{tikzpicture}
\end{center}\caption{Campo de \'indice $I(v,a)=+2$.}\label{figura35}
\end{figure}
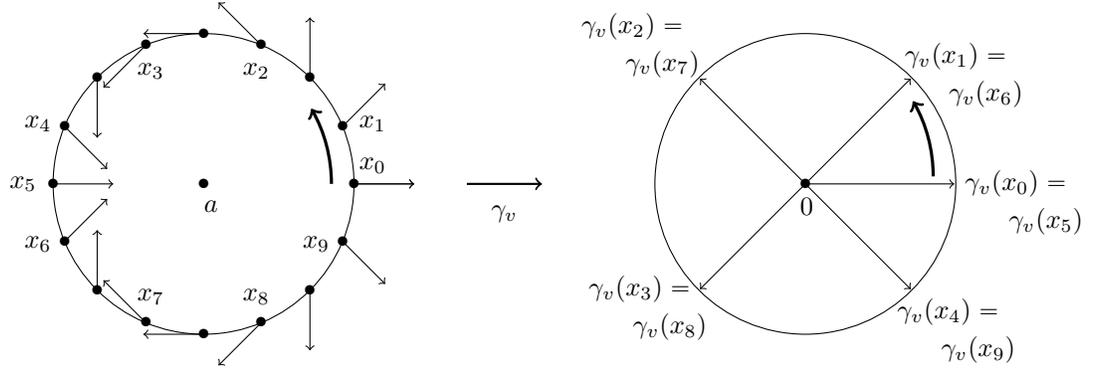

\begin{remark}\label{remarque}
Observamos que se um campo de vetores $v$ definido sobre a borda ${\Sp}^1_a$ de ${B}_a$
se estende sem singularidades dentro de ${B}_a$, ent\~ao o seu \'indice $I(v,a)$ vale $0$. Esta observa\c c\~ao ser\'a \'util para o que se segue. 
\end{remark}

\subsection{Singularidades de campos de vetores definidos em $\R^3$}\label{indices2}\hfill 

Na segunda etapa da prova do Teorema de Poincar\'e-Hopf,  vamos usar uma propriedade 
particular do \'indice de campos de vetores particulares em $\R^3$.
O \'indice de um campo de vetores em $\R^3$ define-se da mesma maneira do que o 
\'indice de um campo de vetores em $\R^2$. Observamos que esta defini\c c\~ao 
\'e a mesma em um espa\c co euclideano de dimens\~ao qualquer: 

Seja $w$ um campo de vetores em $\R^3$ com uma singularidade isolada em $0$. 
Isso significa que podemos considerar uma pequena bola 
$B_0$ centrada em $0$ tal que ao longo da esfera  $\Sp^2_0$, que \'e a borda de $B_0$, 
o campo de vetores $w$ n\~ao tem singularidades.

Para cada ponto $x$ de $\Sp^2_0$, consideramos o vetor 
$$\widetilde w(x) = \frac{w(x)}{\Vert w(x) \Vert}$$ 
em um outro exemplar de $\R^3$, com mesma orienta\c c\~ao e chamado de 
$\widetilde{\R}^3$. 
Aqui  $\Vert w(x) \Vert$ \'e a norma euclidiana de $w(x)$. 

O vetor $\widetilde w(x)$ \'e de tamanho 1 e a sua extremidade, $\gamma_w(x)$, pertence 
\`a esfera ${\Sp}^2$ de raio 1 e de centro \`a origem de $\widetilde{\R}^3$. 

A aplica{\c c}\~ao $\gamma_w: \Sp^2_0 \to {\Sp}^2$ definida por $x \mapsto  \gamma_w(x)$ 
\'e chamada de {\it aplica{\c c}\~ao de Gauss}. 

Como no caso anterior \S\ref{grau2}, 
um ponto $a\in  \Sp^2_0$ tal que a aplica\c c\~ao derivada 
$$d(\gamma_w)_a : T_a( \Sp^2_0) \to T_{(\gamma_w (a)) }\Sp^2$$
tem o posto menor do que 2 \'e chamado de  ponto cr\'itico de $\gamma_w$. 
Os pontos regulares e valores cr\'iticos e regulares est\~ao definidos da mesma maneira
do que \S\ref{grau2}. 

Em um valor regular $y$ de $\gamma_w$, podemos considerar o n\'umero inteiro 
$$ {\rm grau}(\gamma_w;y) = \sum_{x\in \gamma_w^{-1}(y)} {\rm sign} \; d(\gamma_w)_x.$$
Como na se\c c\~ao  \S\ref{grau2},   o n\'umero $ {\rm grau}(\gamma_w;y) $ n\~ao depende do valor 
regular de $\gamma_w$, ele \'e denotado por $ {\rm grau}(\gamma_w)$ e  \'e chamado {\it grau} de $\gamma_w$.
0 \'indice $I(w,0)$ do campo $w$ em $\R^3$ com singularidade isolada em $0$ \'e bem definido como 
o grau de $\gamma_w$.
\begin{lemma}\label{extension}
Seja $v$ um campo de vetores em $\R^2$ com singularidade isolada em $0_{\R^2}$, de \'indice $I(v,0)$ 
e seja $t$ um campo de vetores em $\R$ com singularidade isolada em $0_{\R}$, saindo radialmente do ponto $0$. 
A soma $w = v\oplus t$ \'e um campo de vetores em $\R^3 = \R^2 \oplus \R$ 
tal que $w((x,y),z) = (v(x,y), 0_{\R}) + (0_{\R^2}, t(z))$, com singularidade isolada 
em $0_{\R^3}=(0_{\R^2} , 0_{\R})$ de \'indice igual \`a $I(v,0)$. 
\end{lemma}

Este resultado, j\'a usado em \cite{Sc2},  aparece no livro \cite{BSS} como ``Lemma 2.3.3'', sem prova. A prova que 
providenciamos aqui usa o conceito de grau. Observamos que a hip\'otese sobre o campo $t$ implica
que o \'indice $I(t,0_{\R}) $ vale $+1$. 

\begin{proof}
Denotamos por $B_0$ a bola de centro $0$ e de raio $1$ em $\R^3$. A bola 
$B_0$ \'e homeomorfa \`a $D^2 \times I$, onde $D^2$ \'e um disco de dimens\~ao dois  e $I$ \'e o intervalo $[ -1, +1]$.
A borda de $B_0$, isto \'e, a esfera  $\Sp^2_0 $ \'e homeomorfa \`a lata $(\Sp^1 \times I ) \cup ( D^2 \times (\{ -1 \} \cup \{ +1 \}))$. 

Na lata, temos coordenadas  $((r,\alpha), \lambda)$ tais que: sobre $(\Sp^1 \times I )$, temos $r=1$, $\alpha\in [0, 2\pi[$ e $\lambda \in [-1, +1]$; sobre $D^2 \times 
(\{ -1 \} \cup \{ +1 \})$ temos $r\in [0, r]$, $\alpha \in [0, 2\pi[$ e $\lambda \in \{-1\} \cup \{ +1\}$. 
Observamos que, quando $r=0$, a coordenada $\alpha$ n\~ao \'e definida. 

Construimos um campo de vetores $w$ sobre $(\Sp^1 \times I ) \cup ( D^2 \times 
(\{ -1 \} \cup \{ +1 \}))$ da  seguinte maneira:\\
a) sobre $(\Sp^1 \times I ) $, no ponto $x= (1,\alpha, \lambda)$, o campo $w(x)$ vale $v(1,\alpha) + t( \lambda )$;\\
b) sobre $D^2 \times  \{ -1 \}$ (respectivamente,  $D^2 \times \{ +1 \}$), no ponto 
$x= (r, \alpha, k)$ o campo $w(x)$  \'e $v(r,\alpha) + t(k)$ com $k = -1$ ou $k= +1$,  respectivamente.

O campo $w$ assim definido \'e o que \'e chamado de prolongamento radial de $v$ em \cite{Sc2}. 
O campo $w$ tem uma singularidade isolada na origem  $0_{\R^3}=(0_{\R^2} , 0_{\R})$. 

Consideramos em $\R^3$ a orienta\c c\~ao dada pela orienta\c c\~ao de $\R^2$ seguida pela 
orienta\c c\~ao de $\R$. Observamos que, como o campo de vetores $t$ \'e radial saindo de $0_{\R}$,
este campo conserva a orienta\c c\~ao de $\R$. Ent\~ao, localmente, a orienta\c c\~ao induzida 
por $w$ em $\R^3$ \'e a mesma que a orienta\c c\~ao induzida por $v$ em $\R^2$, 
 seja positiva ou seja negativa   
(veja Observa\c c\~ao \ref{orientation}). 

Em cada ponto $x$ de $\Sp^2_0$ (ou da lata), temos $w(x)\ne 0$. Como no caso de dimens\~ao 2, podemos considerar a aplica\c c\~ao de Gauss 
$$\gamma_w : \Sp^2_0 \to \Sp^2,  \qquad \gamma_w(x) \text{ \'e a  extremidade  do  vetor} 
\widetilde w(x) = \frac{w(x)}{\Vert w(x) \Vert}.$$

Seja $\gamma_v$ a aplica\c c\~ao de Gauss correspondente ao campo $v$,
definida como na se\c cao \ref{grau2}. 
Um valor regular de $\gamma_w$ tamb\'em \'e um valor regular de $\gamma_v$. 
Neste valor,  as orienta\c c\~oes de $\gamma_v$ e $\gamma_w$ coincidem. 
Logo, temos
$$grau(\gamma_w) = grau(\gamma_v).$$
Portanto, os \'indices $I(w,0_{\R^3})=grau(\gamma_w)$ e $I(v,0_{\R^2})=grau(\gamma_v)$ coincidem.
\end{proof}

\section{Teorema de Poincar\'e-Hopf }\hfill

O Teorema de Poincar\'e-Hopf foi provado por Poincar\'e no ano 1885 
no caso de superf\'icies, depois  por Hopf, no ano 1927, para o caso de variedades suaves 
de dimens\~ao maior. 

\begin{theorem}
{\it Seja $v$ um campo de vetores com singularidades isoladas, tangente a uma superf\'icie $\Su$ compacta, suave, sem borda e n\~ao necessariamente orientada. Temos 
$$\sum_{a_i} I(v, a_i) = \chi(\Su),$$
onde os pontos $a_i$ s\~ao os pontos singulares de $v$. }
\end{theorem}

Antes de oferecer uma ideia para a prova, voltamos aos exemplos que j\'a consideramos: um exemplo de campos de vetores 
sobre a esfera $\Sp^2$ \'e dado na Figura \ref{sphere}. Neste exemplo, o campo tem 
singularidades de \'indice $+1$ nos polos norte e sul (campos radiais saindo e entrando, respectivamente, veja Figura \ref{figura28}). Tem-se que a 
caracter\'istica de Euler-Poincar\'e da esfera \'e 2, que tamb\'em \'e a soma dos \'indices nos pontos singulares. 

No toro, existem campos de vetores tangentes sem nenhuma singularidade (veja Figura \ref{tore}).
Ent\~ao a soma $\sum_{a_i} I(v, a_i)$ vale $0$, o que \'e a caracter\'istica de 
Euler-Poincar\'e do toro. 

No plano projetivo $\PP^2$, que n\~ao \'e  orient\'avel, a Figura  \ref{AF} providencia um campo de vetores tangentes, com somente uma singularidade de \'indice $+1$. Isto \'e a caracter\'istica de 
Euler-Poincar\'e do espa\c co projetivo. 

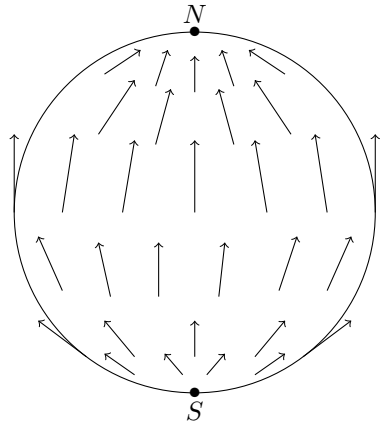
\begin{figure}[h!]
\begin{center}
\begin{tikzpicture}[scale=0.8]

 \draw (0,0) circle (3);
 
 \node at (0,3.3) {$N$};
 \node at (0,3) {$\bullet$};
  \node at (0,-3.3) {$S$};
 \node at (0,-3) {$\bullet$};
 
 \draw [->] (-1.5,2.3) -- (-0.9,2.7);
  \draw [->] (-0.65,2.1) -- (-0.45,2.7);
    \draw [->] (0,2) -- (0,2.6);
      \draw [->] (0.65,2.1) -- (0.45,2.7);
    \draw [->] (1.5,2.3) -- (0.9,2.7);
    
   \draw [->] (-1.6,1.3) -- (-1,2.2);
  \draw [->] (-0.65,1.13) -- (-0.4,2.05);
   \draw [->] (0.65,1.13) -- (0.4,2.05);
 \draw [->] (1.6,1.3) -- (1,2.2);

\draw [->] (-3,0) -- (-3,1.3);
\draw [->] (-2.2,0) -- (-2,1.3);
 \draw [->] (-1.2,0) -- (-1,1.2);
 \draw [->] (0,0) -- (0,1.2);
  \draw [->] (1.2,0) -- (1,1.2);
  \draw [->] (2.2,0) -- (2,1.3);
  \draw [->] (3,0) -- (3,1.3);
  
\draw[<-] (-2.6,-0.4)--  (-2.2,-1.3);
\draw[<-] (-1.6,-0.5)--  (-1.4,-1.4);
\draw[<-] (-0.6,-0.5)--  (-0.6,-1.4);
\draw[<-] (0.5,-0.5)--  (0.4,-1.4);
\draw[<-] (1.7,-0.4)--  (1.4,-1.3);
\draw[<-] (2.6,-0.4)--  (2.2,-1.3);

\draw[<-] (-2.6,-1.8)--  (-1.8,-2.4);
\draw[<-] (-1.5,-1.8)--  (-1,-2.4);
\draw[<-] (0,-1.8)--  (0,-2.4);
\draw[<-] (1.5,-1.8)--  (1,-2.4);
   \draw[<-] (2.6,-1.8)--  (1.8,-2.4);
   
 \draw[<-] (-0.5,-2.35)--  (-0.2,-2.7);
\draw[<-] (-1.5,-2.35)--  (-1,-2.7);
 \draw[<-] (0.5,-2.35)--  (0.2,-2.7);
\draw[<-] (1.5,-2.35)--  (1,-2.7);
    
\end{tikzpicture}
\medskip
\caption{Campo de vetores tangentes a esfera $\Sp^2$. Temos $ I(v,N)=I(v,S)=+1$.}\label{sphere}
\end{center}
\end{figure}

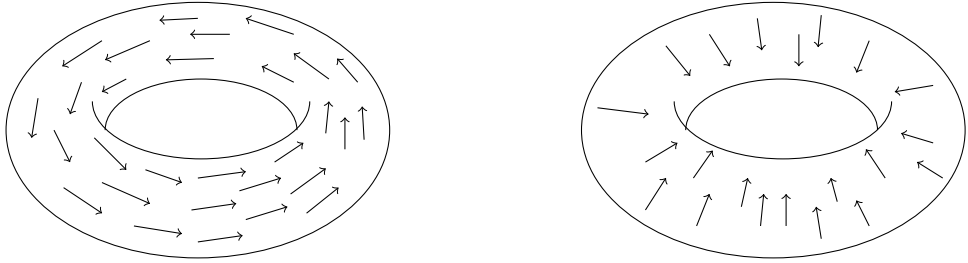
\begin{figure}[h!]
\begin{center}
\begin{tikzpicture}[scale=0.85]




\draw (0,0) arc (0:360: 3 and 2);
\draw (-4.65,0.45) arc (180:360: 1.7 and 0.9);    
\draw (-1.45,0) arc (0:180: 1.5 and 0.8);   


 \draw[->] (-3,-0.75) -- (-2.25,-0.65); 
  \draw[->] (-3.1,-1.25) -- (-2.4,-1.15);
 \draw[->] (-3,-1.75) -- (-2.3,-1.65); 
 
 \draw[->]  (-2.35,-0.95) -- (-1.7,-0.75);
  \draw[->]  (-2.25,-1.4) -- (-1.6,-1.2);

  \draw[->] (-1.8,-0.5) -- (-1.35,-0.2); 
  \draw[->] (-1.55,-1) -- (-1,-0.6);
 \draw[->] (-1.3,-1.3) -- (-0.8,-0.9); 
 
 \draw[->] (-1,-0.05) -- (-0.95,0.45); 
  \draw[->] (-0.7,-0.3) -- (-0.7,0.2);
 \draw[->] (-0.4,-0.15) -- (-0.43,0.37); 
 
  \draw[->] (-0.5,0.75) -- (-0.82,1.12); 
  \draw[->] (-0.95,0.8) -- (-1.5,1.2);
 \draw[->] (-1.5,0.75) -- (-2,1); 

\draw[->] (-1.5,1.5) -- (-2.25,1.75); 
\draw[->] (-2.5,1.5) -- (-3.12,1.5); 
\draw[->] (-2.75,1.12) -- (-3.5,1.1); 

\draw[->] (-3,1.75) -- (-3.6,1.72); 
\draw[->] (-3.75,1.4) -- (-4.45,1.1); 

\draw[->] (-4.12,0.8) -- (-4.5,0.6); 
\draw[->] (-4.5,1.4) -- (-5.12,1); 

\draw[->] (-4.82,0.75) -- (-5,0.25); 
\draw[->] (-5.5,0.5) -- (-5.6,-0.12); 

\draw[->] (-5.25,0) -- (-5,-0.5); 

\draw[->] (-4.62,-0.12) -- (-4.12,-0.62); 
\draw[->] (-5.1,-0.9) -- (-4.5,-1.3); 

\draw[->] (-3.82,-0.62) -- (-3.25,-0.82); 
\draw[->] (-4.5,-0.82) -- (-3.75,-1.15); 
\draw[->] (-4,-1.5) -- (-3.25,-1.62);

 
\draw (9,0) arc (0:360: 3 and 2);
\draw (4.45,0.45) arc (180:360: 1.7 and 0.9);    
\draw (7.63,0) arc (0:180: 1.5 and 0.8); 


 \draw[->] (4.8,-1.5) -- (5,-1); 
 \draw[->] (5.5,-1.2) -- (5.6,-0.75); 
 \draw[->] (5.8,-1.5) -- (5.85,-1);
 \draw[->] (6.2,-1.5) -- (6.2,-1);
  \draw[->] (6.75,-1.7) -- (6.67,-1.2);
  
   \draw[->] (7,-1.12) -- (6.9,-0.75); 
 \draw[->] (7.5,-1.5) -- (7.3,-1.1); 
 \draw[->] (7.75,-0.75) -- (7.45,-0.3);
 \draw[->] (8.65,-0.75) -- (8.25,-0.5);
  \draw[->] (8.5,-0.2) -- (8,-0.05);

  \draw[->] (8.5,0.7) -- (7.9,0.6);
    \draw[->] (7.5,1.4) -- (7.3,0.9);
     \draw[->] (6.75,1.8) -- (6.7,1.3);
         \draw[->] (6.4,1.5) -- (6.4,1);
         
   \draw[->] (5.75,1.75) -- (5.82,1.25); 
 \draw[->] (5,1.5) -- (5.32,1); 
 \draw[->] (4.32,1.32) -- (4.7,0.85);
 \draw[->] (3.25,0.35) -- (4.05,0.25);
  \draw[->] (4,-0.5) -- (4.5,-0.2);

 \draw[->] (4,-1.25) -- (4.32,-0.75);
 \draw[->] (4.75,-0.75) -- (5.05,-0.32);
           
\end{tikzpicture}
\caption {Campos de vetores tangentes ao toro $\T$ : 
Campo tangente as ``paralelas''  e campo tangente as ``meridianas''. N\~ao h\'a singularidades.}\label{tore}\end{center}
\end{figure}

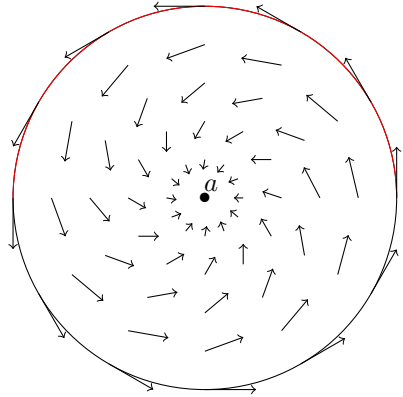
\begin{figure}[h!]
\begin{center}
\begin{tikzpicture}[scale=0.85]

 \node at (0,0) {$\bullet$};
  \node at (0.1,0.2) {$a$};

 \draw (0,0) circle (3);
   \draw [red] (3,0) arc (0:180:3);
 \draw[->] (3,0) -- (3,0.8);
 

\path (1*30:3) coordinate (1);
\begin{scope}[shift={(1)},rotate=90+1*30]; 
\draw[-> ] (1) -- (0.8,0);
  \end{scope}
 
\path (2*30:3) coordinate (2);
\begin{scope}[shift={(2)},rotate=90+2*30]; 
\draw[-> ] (2) -- (0.8,0);
  \end{scope}

\path (3*30:3) coordinate (3);
\begin{scope}[shift={(3)},rotate=90+3*30]; 
\draw[-> ] (3) -- (0.8,0);
  \end{scope}

\path (4*30:3) coordinate (4);
\begin{scope}[shift={(4)},rotate=90+4*30]; 
\draw[-> ] (4) -- (0.8,0);
  \end{scope}

\path (5*30:3) coordinate (5);
\begin{scope}[shift={(5)},rotate=90+5*30]; 
\draw[-> ] (5) -- (0.8,0);
  \end{scope}

\path (6*30:3) coordinate (6);
\begin{scope}[shift={(6)},rotate=90+6*30]; 
\draw[-> ] (6) -- (0.8,0);
  \end{scope}

\path (7*30:3) coordinate (7);
\begin{scope}[shift={(7)},rotate=90+7*30]; 
\draw[-> ] (7) -- (0.8,0);
  \end{scope}

\path (8*30:3) coordinate (8);
\begin{scope}[shift={(8)},rotate=90+8*30]; 
\draw[-> ] (8) -- (0.8,0);
  \end{scope}

\path (9*30:3) coordinate (9);
\begin{scope}[shift={(9)},rotate=90+9*30]; 
\draw[-> ] (9) -- (0.8,0);
  \end{scope}

\path (10*30:3) coordinate (10);
\begin{scope}[shift={(10)},rotate=90+10*30]; 
\draw[-> ] (10) -- (0.8,0);
  \end{scope}

\path (11*30:3) coordinate (11);
\begin{scope}[shift={(11)},rotate=90+11*30]; 
\draw[-> ] (11) -- (0.8,0);
  \end{scope}


 \draw[->] (2.4,0) -- (2.25,0.64);
 
\path (1*30:2.4) coordinate (1);
\begin{scope}[shift={(1)},rotate=110+1*30]; 
\draw[-> ] (1) -- (0.64,0);
  \end{scope}
  
\path (2*30:2.4) coordinate (2);
\begin{scope}[shift={(2)},rotate=110+2*30]; 
\draw[-> ] (2) -- (0.64,0);
  \end{scope}

\path (3*30:2.4) coordinate (3);
\begin{scope}[shift={(3)},rotate=110+3*30]; 
\draw[-> ] (3) -- (0.64,0);
  \end{scope}

\path (4*30:2.4) coordinate (4);
\begin{scope}[shift={(4)},rotate=110+4*30]; 
\draw[-> ] (4) -- (0.64,0);
  \end{scope}

\path (5*30:2.4) coordinate (5);
\begin{scope}[shift={(5)},rotate=110+5*30]; 
\draw[-> ] (5) -- (0.64,0);
  \end{scope}

\path (6*30:2.4) coordinate (6);
\begin{scope}[shift={(6)},rotate=110+6*30]; 
\draw[-> ] (6) -- (0.64,0);
  \end{scope}

\path (7*30:2.4) coordinate (7);
\begin{scope}[shift={(7)},rotate=110+7*30]; 
\draw[-> ] (7) -- (0.64,0);
  \end{scope}

\path (8*30:2.4) coordinate (8);
\begin{scope}[shift={(8)},rotate=110+8*30]; 
\draw[-> ] (8) -- (0.64,0);
  \end{scope}

\path (9*30:2.4) coordinate (9);
\begin{scope}[shift={(9)},rotate=110+9*30]; 
\draw[-> ] (9) -- (0.64,0);
  \end{scope}

\path (10*30:2.4) coordinate (10);
\begin{scope}[shift={(10)},rotate=110+10*30]; 
\draw[-> ] (10) -- (0.64,0);
  \end{scope}

\path (11*30:2.4) coordinate (11);
\begin{scope}[shift={(11)},rotate=105+11*30]; 
\draw[-> ] (11) -- (0.64,0);
  \end{scope}


 \draw[->] (1.8,0) -- (1.55,0.48);
 
\path (1*30:1.8) coordinate (1);
\begin{scope}[shift={(1)},rotate=130+1*30]; 
\draw[-> ] (1) -- (0.48,0);
  \end{scope}
  
\path (2*30:1.8) coordinate (2);
\begin{scope}[shift={(2)},rotate=130+2*30]; 
\draw[-> ] (2) -- (0.48,0);
  \end{scope}

\path (3*30:1.8) coordinate (3);
\begin{scope}[shift={(3)},rotate=130+3*30]; 
\draw[-> ] (3) -- (0.48,0);
  \end{scope}

\path (4*30:1.8) coordinate (4);
\begin{scope}[shift={(4)},rotate=130+4*30]; 
\draw[-> ] (4) -- (0.48,0);
  \end{scope}

\path (5*30:1.8) coordinate (5);
\begin{scope}[shift={(5)},rotate=130+5*30]; 
\draw[-> ] (5) -- (0.48,0);
  \end{scope}

\path (6*30:1.8) coordinate (6);
\begin{scope}[shift={(6)},rotate=130+6*30]; 
\draw[-> ] (6) -- (0.48,0);
  \end{scope}

\path (7*30:1.8) coordinate (7);
\begin{scope}[shift={(7)},rotate=130+7*30]; 
\draw[-> ] (7) -- (0.48,0);
  \end{scope}

\path (8*30:1.8) coordinate (8);
\begin{scope}[shift={(8)},rotate=130+8*30]; 
\draw[-> ] (8) -- (0.48,0);
  \end{scope}

\path (9*30:1.8) coordinate (9);
\begin{scope}[shift={(9)},rotate=130+9*30]; 
\draw[-> ] (9) -- (0.48,0);
  \end{scope}

\path (10*30:1.8) coordinate (10);
\begin{scope}[shift={(10)},rotate=130+10*30]; 
\draw[-> ] (10) -- (0.48,0);
  \end{scope}

\path (11*30:1.8) coordinate (11);
\begin{scope}[shift={(11)},rotate=130+11*30]; 
\draw[-> ] (11) -- (0.48,0);
  \end{scope}


 \draw[->] (1.2,0) -- (0.9,0.1);
 
\path (1*30:1.2) coordinate (1);
\begin{scope}[shift={(1)},rotate=150+1*30]; 
\draw[-> ] (1) -- (0.32,0);
  \end{scope}
  
\path (2*30:1.2) coordinate (2);
\begin{scope}[shift={(2)},rotate=150+2*30]; 
\draw[-> ] (2) -- (0.32,0);
  \end{scope}

\path (3*30:1.2) coordinate (3);
\begin{scope}[shift={(3)},rotate=150+3*30]; 
\draw[-> ] (3) -- (0.32,0);
  \end{scope}

\path (4*30:1.2) coordinate (4);
\begin{scope}[shift={(4)},rotate=150+4*30]; 
\draw[-> ] (4) -- (0.32,0);
  \end{scope}

\path (5*30:1.2) coordinate (5);
\begin{scope}[shift={(5)},rotate=150+5*30]; 
\draw[-> ] (5) -- (0.32,0);
  \end{scope}

\path (6*30:1.2) coordinate (6);
\begin{scope}[shift={(6)},rotate=150+6*30]; 
\draw[-> ] (6) -- (0.32,0);
  \end{scope}

\path (7*30:1.2) coordinate (7);
\begin{scope}[shift={(7)},rotate=150+7*30]; 
\draw[-> ] (7) -- (0.32,0);
  \end{scope}

\path (8*30:1.2) coordinate (8);
\begin{scope}[shift={(8)},rotate=150+8*30]; 
\draw[-> ] (8) -- (0.32,0);
  \end{scope}

\path (9*30:1.2) coordinate (9);
\begin{scope}[shift={(9)},rotate=150+9*30]; 
\draw[-> ] (9) -- (0.32,0);
  \end{scope}

\path (10*30:1.2) coordinate (10);
\begin{scope}[shift={(10)},rotate=150+10*30]; 
\draw[-> ] (10) -- (0.32,0);
  \end{scope}

\path (11*30:1.2) coordinate (11);
\begin{scope}[shift={(11)},rotate=150+11*30]; 
\draw[-> ] (11) -- (0.32,0);
  \end{scope}


 \draw[->] (0.6,0) -- (0.455,0);
 
\path (1*30:0.6) coordinate (1);
\begin{scope}[shift={(1)},rotate=170+1*30]; 
\draw[-> ] (1) -- (0.16,0);
  \end{scope}
  
\path (2*30:0.6) coordinate (2);
\begin{scope}[shift={(2)},rotate=170+2*30]; 
\draw[-> ] (2) -- (0.16,0);
  \end{scope}

\path (3*30:0.6) coordinate (3);
\begin{scope}[shift={(3)},rotate=170+3*30]; 
\draw[-> ] (3) -- (0.16,0);
  \end{scope}

\path (4*30:0.6) coordinate (4);
\begin{scope}[shift={(4)},rotate=170+4*30]; 
\draw[-> ] (4) -- (0.16,0);
  \end{scope}

\path (5*30:0.6) coordinate (5);
\begin{scope}[shift={(5)},rotate=170+5*30]; 
\draw[-> ] (5) -- (0.16,0);
  \end{scope}

\path (6*30:0.6) coordinate (6);
\begin{scope}[shift={(6)},rotate=170+6*30]; 
\draw[-> ] (6) -- (0.16,0);
  \end{scope}

\path (7*30:0.6) coordinate (7);
\begin{scope}[shift={(7)},rotate=170+7*30]; 
\draw[-> ] (7) -- (0.16,0);
  \end{scope}

\path (8*30:0.6) coordinate (8);
\begin{scope}[shift={(8)},rotate=170+8*30]; 
\draw[-> ] (8) -- (0.16,0);
  \end{scope}

\path (9*30:0.6) coordinate (9);
\begin{scope}[shift={(9)},rotate=170+9*30]; 
\draw[-> ] (9) -- (0.16,0);
  \end{scope}

\path (10*30:0.6) coordinate (10);
\begin{scope}[shift={(10)},rotate=170+10*30]; 
\draw[-> ] (10) -- (0.16,0);
  \end{scope}

\path (11*30:0.6) coordinate (11);
\begin{scope}[shift={(11)},rotate=170+11*30]; 
\draw[-> ] (11) -- (0.16,0);
  \end{scope}

\end{tikzpicture}
\end{center}
\caption{Campo de vetores tangentes ao plano projetivo $\PP^2$. Aqui, o plano projetivo \'e representado como um disco onde os pontos diametralmente opostos da borda s\~ao identificados (veja Figura \ref{figura17}). Temos $I(v,a)=+1$ (veja Figura \ref{figura28}). }\label{AF}.
\end{figure}

\pagebreak
Mostraremos o Teorema de Poincar\'e-Hopf no caso n\~ao-orient\'avel depois de estudar o caso orientado. 

\subsection{Caso orient\'avel}

A prova do Teorema de Poincar\'e-Hopf no caso orientado pode ser feita em duas etapas:

a) Exibir um campo de vetores $v$ tal que 
$$\sum_{a_i} I(v, a_i) = \chi(\Su).$$

b) Provar que a soma $\sum_{a_i} I(v, a_i) $ n\~ao depende do campo de vetores $v$ tangente a $\Su$ com singularidades $a_i$, nem do n\'umero de singularidades, isto \'e, 
para quaisquer campos de vetores $v$ e $v'$ com singularidades isoladas sobre $\Su$, ent\~ao 
$$\sum_{a_i} I(v, a_i) = \sum_{b_j} I(v', b_j),$$
onde $a_i$ e $b_j$ s\~ao as singularidades dos campos $v$ e $v'$, respectivamente.

\subsubsection{Primeira etapa}\hfill 

A primeira etapa exibe um campo de vetores com uma constru{\c c}\~ao bonita, devido a H. Hopf, 
que vamos descrever com detalhes. 

Seja $K$ uma triangula{\c c}\~ao de $\Su$. Para cada simplexo de dimens\~ao 2, vamos construir um campo, chamado de {campo de Hopf}, como seguinte:

Para cada simplexo $\sigma$ de $K$, escolhemos um baricentro denotado por $\widehat{\sigma}$. 

Um simplexo de dimens\~ao $n$, onde $n = 0, 1, 2$, ser\'a denotado por $\sigma^n$, e seu baricentro  denotado por $\widehat{\sigma}^n$.

\begin{figure}[h!]
\begin{center}
\begin{tikzpicture} [scale=1.2]

\draw (0,0) -- (6, 0) -- (2, 4) -- (0, 0);
\draw [dashed] (2.7, 1.2) -- (0,0);
\draw [dashed](2.7, 1.2) -- (6,0);
\draw [dashed] (2.7, 1.2) -- (2,4);
\draw [dashed] (2.7, 1.2) -- (4,2);
\draw [dashed] (2.7, 1.2) -- (1,2);
\draw [dashed] (2.7, 1.2) -- (3,0);
\node at  (0, 0)[below] {$\hat{\sigma}^0$};
\node at  (4, 2)[right] {$\hat{\sigma}^1$};
\node at  (2, 4)[above] {$\hat{\sigma}^0$};
\node at  (6,0)[right] {$\hat{\sigma}^0$};
\node at  (1, 2)[left] {$\hat{\sigma}^1$};
\node at  (3, 0)[below] {$\hat{\sigma}^1$};
\node at (2.9, 1.25) [above] {$\hat{\sigma}^2$};
\node at (2.7, 1.2)  {$\bullet$};

\end{tikzpicture}
\end{center}
\caption{Baricentros e decomposi\c c\~ao baric\^entrica.}\label{AA}
\end{figure}
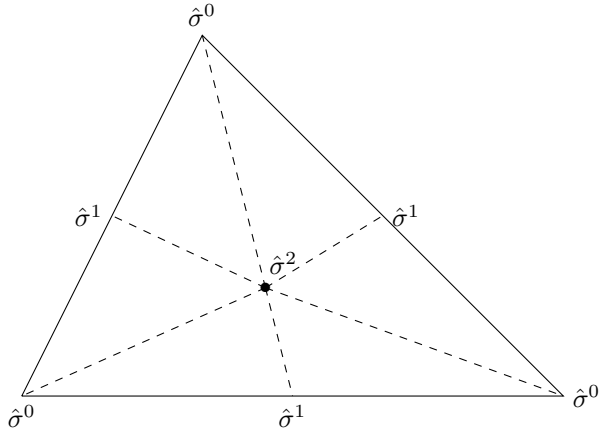

A escolha de baricentros permite definir a decomposi{\c c}\~ao baric\^entrica de $K$, denotada por $K'$, cujos v\'ertices s\~ao todos baricentros $\widehat{\sigma}^0$,
 $\widehat{\sigma}^1$ e $\widehat{\sigma}^2$, os segmentos ser\~ao denotados por $[ \widehat{\sigma}^i, \widehat{\sigma}^j]$ com $i < j$ e tri\^angulos por 
$[ \widehat{\sigma}^0, \widehat{\sigma}^1, \widehat{\sigma}^2]$. 

As singularidades do campo $v$ ser\~ao  todos baricentros $\widehat{\sigma}^i$, $i = 0, 1, 2$, 
e somente estes pontos. O campo ser\'a tangente aos segmentos $[ \widehat{\sigma}^i, \widehat{\sigma}^j]$ 
da decomposi{\c c}\~ao baric\^entrica,  saindo de $\widehat{\sigma}^i$ e indo na dire{\c c}\~ao 
de $\widehat{\sigma}^j$ ($i < j$). Como o campo $v$ se anula nos pontos $\widehat{\sigma}^i$ 
e $\widehat{\sigma}^j$, ele est\'a crescendo saindo de $\widehat{\sigma}^i$ 
at\'e um  m\'aximo e depois 
ele decresce at\'e se anular no ponto $\widehat{\sigma}^j$ 
(Figura \ref{figura32}).

\begin{figure}[h!] 
\begin{tikzpicture} 

\draw (0,0) -- (6,2);
\node at  (0, 0)[above] {$\hat{\sigma}^i$};
\node at  (6, 2)[above] {$\hat{\sigma}^j$};
\node at (0,0)  {$\bullet$};
\node at (6,2)  {$\bullet$};

\draw[-> ,very thick] (0.38,0.14) -- (0.7,0.25);
\draw[-> ,very thick] (1.1,0.4) -- (1.6,0.56);
\draw[-> ,very thick] (2,0.71) -- (2.7,0.95);

\draw[-> ,very thick] (3,1) -- (4,1.333);
\draw[-> ,very thick] (4.5,1.49) -- (5.2,1.745);
\draw[-> ,very thick] (5.45,1.83) -- (5.8,1.95);

\end{tikzpicture}
\caption{O campo de Hopf ao longo de $[ \widehat{\sigma}^i, \widehat{\sigma}^j]$.}\label{figura32}
\end{figure}

No tri\^angulo  $[ \widehat{\sigma}^0, \widehat{\sigma}^1, \widehat{\sigma}^2]$, o campo \'e constru\'ido de tal maneira a ser cont\'inuo e satisfazer as condi{\c c}\~oes anteriores  (Figura \ref{figura33}).

\begin{figure}[h!]
\begin{center}

\begin{tikzpicture}[scale=1.2]


\draw [dashed](0,0) -- (5,5);
\draw (0,0) -- (6,1);
\draw [dashed](5,5) -- (6,1);

\node at  (0, 0)[left] {$\hat{\sigma}^0$};
\node at  (5,5)[right] {$\hat{\sigma}^2$};
\node at  (6,1)[right] {$\hat{\sigma}^1$};

\node at (0,0)  {$\bullet$};
\node at (5,5)  {$\bullet$};
\node at (6,1)  {$\bullet$};

\draw[thick,dotted] 
  (0,0) 
    .. controls  (4,2.5) .. 
  (5,5); 
  
  \draw[thick,dotted] 
  (0,0) 
    .. controls  (4.7,1.5) .. 
  (5,5);   

 \draw[thick,dotted] 
  (0,0) 
    .. controls  (5.7,1) .. 
  (5,5);   
  
   \draw [red, very thick] (5.73,2) arc (95:192:1);


\draw[-> ,very thick] (0.35,0.33) -- (0.6,0.6);
\draw[-> ,very thick] (0.38,0.18) -- (0.7,0.33);
\draw[-> ,very thick] (0.42,0.08) -- (0.8,0.155);

\draw[-> ,very thick] (0.85,0.85) -- (1.2,1.2);
\draw[-> ,very thick] (1,0.62) -- (1.4,0.83);
\draw[-> ,very thick] (1,0.3) -- (1.45,0.45);
\draw[-> ,very thick] (1.05,0.17) -- (1.5,0.27);

\draw[-> ,very thick] (1.4,1.4) -- (1.85,1.85);
\draw[-> ,very thick] (1.65,1.03) -- (2.2,1.33);
\draw[-> ,very thick] (2,0.645) -- (2.55,0.8);
\draw[-> ,very thick] (2,0.35) -- (2.6,0.43);

\draw[-> ,very thick] (2.4,2.4) -- (3,3);
\draw[-> ,very thick] (2.8,1.8) -- (3.5,2.21);
\draw[-> ,very thick] (3.4,1.2) -- (4.1,1.5);
\draw[-> ,very thick] (3.6,0.73) -- (4.35,0.9);
\draw[-> ,very thick] (3.2,0.53) -- (3.8,0.62);

\draw[-> ,very thick] (3.4,3.4) -- (3.85,3.85);
\draw[-> ,very thick] (3.9,2.8) -- (4.3,3.23);
\draw[-> ,very thick] (4.4,2) -- (4.7,2.5);
\draw[-> ,very thick] (4.9,1.38) -- (5.2,1.65);
\draw[-> ,very thick] (4.82,1.1) -- (5.21,1.25);
\draw[-> ,very thick] (4.82,0.8) -- (5.3,0.87);

\draw[-> ,very thick] (4.2,4.2) -- (4.45,4.45);
\draw[-> ,very thick] (4.5,3.8) -- (4.74,4.23);
\draw[-> ,very thick] (4.85,3.4) -- (4.95,3.8);
\draw[-> ,very thick] (4.95,4.2) -- (4.97,4.5);
\draw[-> ,very thick] (5.3,3) -- (5.3,3.37);
\draw[-> ,very thick] (5.2,1.8) -- (5.35,2.2);
\draw[-> ,very thick] (5.5,1.95) -- (5.49,2.3);

\draw[-> ,very thick] (5.3,3.8) -- (5.2,4.1);
\draw[-> ,very thick] (5.55,2.8) -- (5.45,3.17);
\draw[-> ,very thick] (5.75,2) -- (5.67,2.35);

\draw[-> ,very thick] (5.4,1.4) -- (5.56,1.7);

\draw[-> ,very thick] (5.3,1.08) -- (5.56,1.2);
\draw[-> ,very thick] (5.65,1.31) -- (5.67,1.58);
\draw[-> ,very thick] (5.6,0.95) -- (5.8,0.98);
\draw[-> ,very thick] (5.9,1.4) -- (5.85,1.6);

\end{tikzpicture}
\end{center}
\caption{O campo de Hopf no tri\^angulo  $[ \widehat{\sigma}^0, \widehat{\sigma}^1, \widehat{\sigma}^2].$ 
}\label{figura33}
\end{figure}

O campo de Hopf,  a partir de sua constru{\c c}\~ao, admite como singularidades os baricentros de todos os simplexos de $K$. 
\goodbreak

A situa{\c c}\~ao local \'e a mesma para todos baricentros $\widehat{\sigma}^0$
(Figura \ref{figura88}).

Isto \'e, o \'indice  $I(v, \widehat{\sigma}^0) = +1$ (veja Figura \ref{figura28}) para todo baricentro $\widehat{\sigma}^0$. 


\begin{figure}[h!]
\begin{center}
\begin{tikzpicture}[scale=2]

\draw [dashed] (3,2) -- (4.8,3.8);
\draw (3,2) -- (5,2);
\draw [dashed] (3,2) -- (3.8,1.2);
\draw (3,2) -- (2.4,1); 
\draw [dashed] (3,2) -- (1,1.4);
\draw (3,2) -- (2,4);

\draw[-> ,very thick] (2.8,2.4) -- (2.65,2.7);
\draw[-> ,very thick] (2.8,2.8) -- (2.65,3.2);
\draw[-> ,very thick] (3,2.6) -- (3,3);
\draw[-> ,very thick] (3.2,2.6) -- (3.27,2.95);

\draw[-> ,very thick] (3.4,2.4) -- (3.6,2.6);
\draw[-> ,very thick] (3.8,2.6) -- (4.15,2.87);
\draw[-> ,very thick] (3.4,2.2) -- (3.75,2.4);
\draw[-> ,very thick] (3.8,2.2) -- (4.2,2.3);
\draw[-> ,very thick] (3.6,2) -- (4,2);

\draw[-> ,very thick] (3,1.6) -- (3,1.2);
\draw[-> ,very thick] (3.2,1.8) -- (3.4,1.6); 
\draw[-> ,very thick] (3.2,1.6) -- (3.4,1.3);
\draw[-> ,very thick] (3.6,1.6) -- (3.9,1.4);
\draw[-> ,very thick] (3.6,1.8) -- (4,1.63);

\draw[-> ,very thick] (2.8,1.8) -- (2.6,1.6);
\draw[-> ,very thick] (2.8,1.6) -- (2.6,1.3); 
\draw[-> ,very thick] (2.4,1.6) -- (2.1,1.4);
\draw[-> ,very thick] (2.4,1.8) -- (2,1.68);
\draw[-> ,very thick] (2.6,2) -- (2.2,2);

\draw[-> ,very thick] (2.6,2.2) -- (2.25,2.4);
\draw[-> ,very thick] (2.2,2.2) -- (1.8,2.3);
\draw[-> ,very thick] (2.4,2.6) -- (2.1,2.9);

\node at  (3.1,2.1) [above] {$\widehat{\sigma}^0$};

\end{tikzpicture}
\end{center}
\caption{O campo de Hopf em um baricentro $\widehat{\sigma}^0$.}\label{figura88}
\end{figure}

Da mesma maneira, para todo baricentro $\widehat{\sigma}^1$, temos a situa{\c c}\~ao local
da Figura \ref{figura20}. 


\begin{figure}[h!]
\begin{center}
\begin{tikzpicture}[scale=2]


\draw [dashed](3,2) -- (3.5,0.5);
\draw [dashed] (3,2) -- (2.2,3.6);
\draw [thick] (1.5,1.4) -- (4.5,2.6);

\draw[-> ,very thick] (2.8,2.4) -- (2.65,2.7);
\draw[-> ,very thick] (2.8,2.8) -- (2.6,3.2);
\draw[-> ,very thick] (3.2,2.4) -- (2.95,2.6);
\draw[-> ,very thick] (3,3) -- (2.8,3.4);

\draw[-> ,very thick] (3.4,2.6) -- (3.05,2.8);
\draw[-> ,very thick] (3.8,2.4) -- (3.4,2.35);
\draw[-> ,very thick] (3.9,2.56) -- (3.58,2.57);

\draw[-> ,very thick] (4.2,2) -- (3.9,1.87);
\draw[-> ,very thick] (3.6,1.2) -- (3.68,1);

\draw[-> ,very thick] (3.9,2.15) -- (3.6,2);
\draw[-> ,very thick] (3.8,1.95) -- (3.55,1.78);
\draw[-> ,very thick] (3.4,1.8) -- (3.3,1.6);
\draw[-> ,very thick] (3.4,1.4) -- (3.45,1.2);
\draw[-> ,very thick] (3.62,1.6) -- (3.55,1.4);

\draw[-> ,very thick] (3.12,1.6) -- (3.21,1.4);
\draw[-> ,very thick] (2.2,1.7) -- (2.5,1.8);
\draw[-> ,very thick] (2.2,1.6) -- (2.5,1.65);
\draw[-> ,very thick] (2.65,1.6) -- (2.82,1.5);

\draw[-> ,very thick] (2.28,1.42) -- (2.6,1.4);
\draw[-> ,very thick] (2.87,1.6) -- (3.04,1.4);
\draw[-> ,very thick] (2.8,1.2) -- (3.07,1);
\draw[-> ,very thick] (3.3,1.1) -- (3.4,0.8);

\draw[-> ,very thick] (1.8,1.9) -- (2.1,2.1);
\draw[-> ,very thick] (2,1.8) -- (2.37,1.97);
\draw[-> ,very thick] (2.52,1.95) -- (2.7,2.2);
\draw[-> ,very thick] (2.6,2.4) -- (2.52,2.7);
\draw[-> ,very thick] (2.3,2.2) -- (2.38,2.43);
\draw[-> ,very thick] (2.3,2.6) -- (2.27,2.93);

\draw[-> ,very thick] (2.56,2.9) -- (2.35,3.3);

\node at  (3.2,2) [above] {$\widehat{\sigma}^1$};
\node at  (4.4,2.6) [below] {${\sigma}^1$};

\end{tikzpicture}
\end{center}
\caption{O campo de Hopf em um baricentro $\widehat{\sigma}^1$.}\label{figura20}
\end{figure}


\begin{figure}[h!]
\begin{center}
\begin{tikzpicture}[scale=2]

\draw [dashed] (3,2) -- (4.8,3.8);
\draw [dashed] (3,2) -- (5,2);
\draw [dashed] (3,2) -- (3.8,1.2);
\draw [dashed] (3,2) -- (2.4,1); 
\draw [dashed] (3,2) -- (1,1.4);
\draw [dashed] (3,2) -- (2,4);

\draw[<- ,very thick] (2.8,2.4) -- (2.65,2.7);
\draw[<- ,very thick] (2.8,2.8) -- (2.65,3.2);
\draw[<- ,very thick] (3,2.6) -- (3,3);
\draw[<- ,very thick] (3.2,2.6) -- (3.27,2.95);

\draw[<- ,very thick] (3.4,2.4) -- (3.6,2.6);
\draw[<- ,very thick] (3.8,2.6) -- (4.15,2.87);
\draw[<- ,very thick] (3.4,2.2) -- (3.75,2.4);
\draw[<- ,very thick] (3.8,2.2) -- (4.2,2.3);
\draw[<- ,very thick] (3.6,2) -- (4,2);

\draw[<- ,very thick] (3,1.6) -- (3,1.2);
\draw[<- ,very thick] (3.2,1.8) -- (3.4,1.6); 
\draw[<- ,very thick] (3.2,1.6) -- (3.4,1.3);
\draw[<- ,very thick] (3.6,1.6) -- (3.9,1.4);
\draw[<- ,very thick] (3.6,1.8) -- (4,1.63);

\draw[<- ,very thick] (2.8,1.8) -- (2.6,1.6);
\draw[<- ,very thick] (2.8,1.6) -- (2.6,1.3); 
\draw[<- ,very thick] (2.4,1.6) -- (2.1,1.4);
\draw[<- ,very thick] (2.4,1.8) -- (2,1.68);
\draw[<- ,very thick] (2.6,2) -- (2.2,2);

\draw[<- ,very thick] (2.6,2.2) -- (2.25,2.4);
\draw[<- ,very thick] (2.2,2.2) -- (1.8,2.3);
\draw[<- ,very thick] (2.4,2.6) -- (2.1,2.9);

\node at  (3.1,2.1) [above] {$\widehat{\sigma}^2$};


\end{tikzpicture}
\end{center}
\caption{O campo de Hopf em um baricentro $\widehat{\sigma}^2$.}\label{figura21}
\end{figure}

Isto \'e, o \'indice $I(v, \widehat{\sigma}^2) = +1$ (campo radial entrando, veja Figura \ref{figura28}) para todo baricentro $\widehat{\sigma}^2$.

Agora, vamos calcular a soma 
$$\sum_{a_i} I(v, a_i)$$
para todos os pontos singulares $a_i$ do campo de vetores $v$. Os pontos singulares s\~ao: 

+ os baricentros $\widehat{\sigma}^0_i$ dos v\'ertices $\sigma^0_i$,

+ os baricentros $\widehat{\sigma}^1_j$ dos segmentos $\sigma^1_j$,

+ os baricentros $\widehat{\sigma}^2_k$ dos tri\^angulos $\sigma^2_k$.

Ent\~ao, tem-se 
$$\sum_{a_i} I(v, a_i) = \sum_{{\sigma}^0_i} I(v, \widehat{\sigma}^0_i) + \sum_{{\sigma}^1_j} I(v, \widehat{\sigma}^1_j) + \sum_{{\sigma}^2_k} I(v, \widehat{\sigma}^2_k).$$
Mas vemos que, para todo $\sigma^0_i$, temos $I(v, \widehat{\sigma}^0_i) = 1$, ent\~ao $\sum_{\sigma^0_i} I(v, \widehat{\sigma}^0_i)$ \'e o n\'umero $n_0$ dos v\'ertices. \\
Da mesma maneira, para todo  $\sigma^1_j$, temos $I(v, \widehat{\sigma}^1_j) = -1$, ent\~ao $\sum_{\sigma^1_j} I(v, \widehat{\sigma}^1_j)$ \'e igual a $-n_1$, onde $n_1$ \'e 
o n\'umero dos segmentos. Enfim, para todos $\sigma^2_k$, temos $I(v, \widehat{\sigma}^2_k) = +1$, ent\~ao $\sum_{\sigma^2_k} I(v, \widehat{\sigma}^2_k)$ \'e igual ao  
 n\'umero $n_2$ dos tri\^angulos. \\
Ent\~ao, temos
$$\sum_{a_i} I(v, a_i) = n_0 - n_1 + n_2 = \chi(\Su),$$
onde $a_i$ s\~ao todos os pontos singulares do campo de vetores $v$. 

\bigskip
\subsubsection{Segunda etapa}\hfill

A segunda etapa \'e a mais delicada e h\'a v\'arias maneiras de trabalhar. Vamos providenciar as  ideias 
de Heinz Hopf \cite{Ho1,Ho2}, revisitadas por John Milnor \cite{Mi} e Marie-H\'el\`ene Schwartz \cite{Sc1}. 

Sendo um  campo $v$ de vetores tangentes a $\Su$, com singularidades isoladas $a_i$, vamos estender o campo em uma vizinhan\c ca $X_\varepsilon$ de $\Su$, com as mesmas 
singularidades e os mesmos \'indices. Isto \'e chamado de {\it prolongamento radial}. 
A vizinhan\c ca $X_\varepsilon$ tem borda $\partial X_\varepsilon$ sobre a qual \'e definida a aplica\c c\~ao de Gauss 
$$n : \partial X_\varepsilon \to \Sp^2$$
que associa a cada ponto $y\in \partial X_\varepsilon$ o vetor normal unit\'ario 
$n(y)$ saindo de $X_\varepsilon$ no ponto $y$.
A segunda etapa ser\'a conclu\'ida mostrando que se a soma dos \'indices do campo $v$ \'e igual ao grau 
da aplica\c c\~ao de Gauss, ent\~ao esta soma \'e independente do campo $v$ e do n\'umero de pontos singulares de $v$. 

O m\'etodo de prolongamento radial apareceu pela primeira vez, e ao mesmo tempo, no livro de John Milnor,
``Topology from the differential viewpoint'',  e na constru\c c\~ao de classes caracter\'isticas de variedades 
singulares de Marie-H\'el\`ene Schwartz. Milnor providenciou esta constru\c c\~ao no caso de campos de vetores 
com singularidades n\~ao degeneradas, isto \'e, de \'indice $+1$ ou $-1$, sobre variedades suaves.
Ele n\~ao percebeu que sua constru\c c\~ao vale para uma singularidade isolada de qualquer \'indice. 
Marie-H\'el\`ene Schwartz providenciou a constru\c c\~ao no caso de singularidades isoladas de qualquer \'indice, e sobre variedades singulares, o que \'e muito mais geral e mais delicada. 
\medskip

Agora, vamos \`a constru\c c\~ao. 
Em  primeiro lugar, vamos construir o prolongamento radial de um campo de vetores tangentes
a $\Su$ com singularidades isoladas nos pontos $a_i$. 


Chamamos $X_\varepsilon$ a vizinhan\c ca fechada de $\mathcal{S}$, de ``tamanho'' $\varepsilon$, isto \'e,  o conjunto de todos os pontos $x\in \R^3$ tais que $\Vert x-y \Vert \le \varepsilon$ para 
um ponto $y\in \Su$. Para $\varepsilon$ suficientemente pequeno, $X_\varepsilon$ \'e uma variedade suave de dimens\~ao 3 em $\R^3$, com borda
$\partial X_\varepsilon$.  

O campo de vetores radiais, dentro de $X_\varepsilon$ \'e constru\'ido da  segunite  maneira: para $x\in X_\varepsilon$, seja $r(x)$ 
o ponto de $\Su$ mais pr\'oximo de $x$. 
O vetor $t(x)=x-r(x)$ (em $\R^3$) \'e ortogonal ao espa\c co tangente a $\Su$ no ponto $r(x)$, sen\~ao, $r(x)$ n\~ao seria o ponto de $\Su$ mais pr\'oximo de $x$. 
O vetor $t(x)$ \'e chamado de {\it transversal}.

\begin{center}
\begin{figure}[h!]
\begin{tikzpicture}[scale=0.8]
\draw  (0,0) -- (0,3) -- (1,4) -- (4,4) -- (4,1) -- (3, 0) --(0,0);
\draw[dotted, thick] (0,0) --(1,1) -- (4,1);
\draw[dotted, thick] (1, 1) -- (1, 4);
\draw (0,3) -- (3,3) -- (4,4);
\draw (3,3) -- (3,0); 

\draw[dotted, thick] (4,1) -- (4,-0.6);
\draw[dotted, thick] (0,0) -- (0,-0.6);
\draw[dotted, thick] (3,0) -- (3,-0.6);
\draw (0,-0.6) -- (0,-3) -- (3, -3) -- (3,-0.6);
\draw (3, -3) -- (4, -2) -- (4,-0.6);
\draw (4,2) -- (6.7,2) -- (4.2,-0.6) -- (-2,-0.6) -- (0,1.5);

\node at (-1.3, -0.3) {$\Su$};

\draw[dotted, thick] (1.6, 0.5) -- (1.6, 2.3);
\node at (1.6, 0.5) {$\bullet$};
\node at (1.6, 0.65)[below] {$r(x)$};
\node at (1.6, 2.3) {$\bullet$};
\node at (1.6, 2.3)[left] {$x$};

\fill[fill=blue!20] (0,3) -- (1,4) -- (4,4) -- (3,3) -- (0,3);
\node at (1, 3.4) {$\partial X_\varepsilon$};
\node at (2.5, 1.7) {$X_\varepsilon$};

\draw[->] (2.7 ,3.5) -- (2.7, 4.5);
\node at (2.7, 3.5) {$\bullet$};
\node at (2.7, 3.5)[below] {$y$};
\node at (2.75, 4.5)[left] {$n(y)$};
\end{tikzpicture}
\caption{A vizinhan\c ca $X_\varepsilon$.}\label{vizi2}
\end{figure}
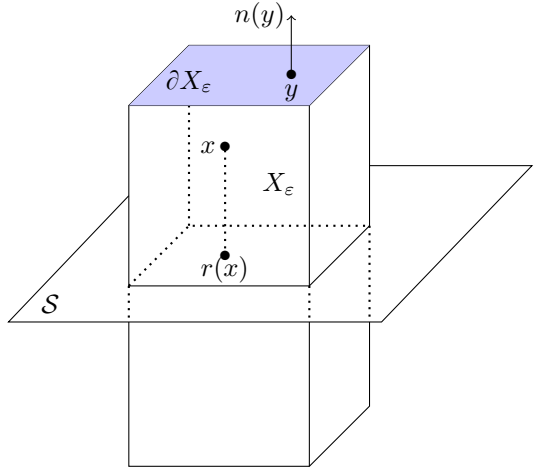
\end{center}


Para $\varepsilon$ suficientemente pequeno, a fun\c c\~ao $r:X_\varepsilon \to \Su$ \'e suave e bem definida. 
Seja $\varphi$ a fun\c c\~ao 
$$\varphi (x) = \Vert x - r(x) \Vert^2,$$
onde ``$\Vert \; \Vert$'' \'e a norma euclidiana em $\R^3$. 
O gradiente de $\varphi$ \'e dado por 
$${\rm grad} \varphi = 2 (x - r(x) ).$$
Em cada ponto $y$ da borda $\partial X_\varepsilon = \varphi^{-1}(\varepsilon^2)$, o vetor normal unit\'ario, saindo de $X_\varepsilon$ \'e dado por 
$$n(y) = \frac{{\rm grad} \varphi} { \Vert {\rm grad} \varphi \Vert} = \frac{(y - r(y) )}{\varepsilon}.$$

Estendemos o campo de vetores $v$ (definido sobre $\Su$) em um campo de vetores $w$ definido sobre $X_\varepsilon$ da seguinte maneira:
$$w(x) = (x - r(x) ) + v(r(x)),$$
ou seja 
$$w(x) = t(x) + v(r(x)).$$
 Lembramos que  $t(x)= x - r(x)$ \'e chamado de campo de vetores {\it transversal}, ele \'e independente de $v$ e 
\'e transversal ao plano tangente a $\Su$. Chamamos 
$v(r(x))$ de campo de vetores {\it paralelos}. Ele \'e  paralelo ao plano tangente a $\Su$ (veja Figura \ref{vizi1}).

Quais s\~ao as propriedades do campo de vetores $w$?

Por um lado, para $y\in \partial X_\varepsilon$, o  produto escalar $w(y). n(y)$ vale $\varepsilon >0$, 
ent\~ao o campo de vetores $w$ sai de $X_\varepsilon$ ao longo  da borda $\partial X_\varepsilon$. Por outro lado, $w$ se anula nos mesmos pontos que $v$. De fato, 
no ponto $x\in X_\varepsilon$, os vetores $(x - r(x) )$ e $ v(r(x))$ s\~ao ortogonais. 

Seja $a$ uma singularidade do campo $v$, ao longo da ``fibra''  
$r^{-1}(a) = \{ x\in X_\varepsilon : r(x)=a \}$ da vizinhan\c ca $X_\varepsilon$, o campo $w$ vale $(x - a )$, 
o que \'e um campo de vetores 
com singularidade isolada no ponto $a$, com \'indice $+1$. 
Assim que demostrado em \cite{Sc2},  o campo $w$ \'e   a 
soma de dois campos de vetores que s\~ao ortogonais: 
o campo $t(x)$ que \'e transversal, nulo ao longo de $X_\varepsilon$ e o campo  paralelo
$v(r(x))$, nulo em todos pontos $z$ tais que $r(z)=a$. Ent\~ao $w$  admite as mesmas singularidades que $v(x)$.  

Como o campo de vetores $t$  sai radialmente do ponto $a$ na fibra $r^{-1}(a)$, 
o Lema \ref{extension} mostra que o \'indice do campo de vetores $w$ 
definido em $X_\varepsilon$, no ponto singular $a$ 
\'e igual ao \'indice do campo de vetores $v$ definido na superf\'icie $\Su$  no 
mesmo ponto singular $a$. Podemos escrever isso como
$$I(w,a;X_\varepsilon) = I(v,a;\Su).$$

\begin{center}
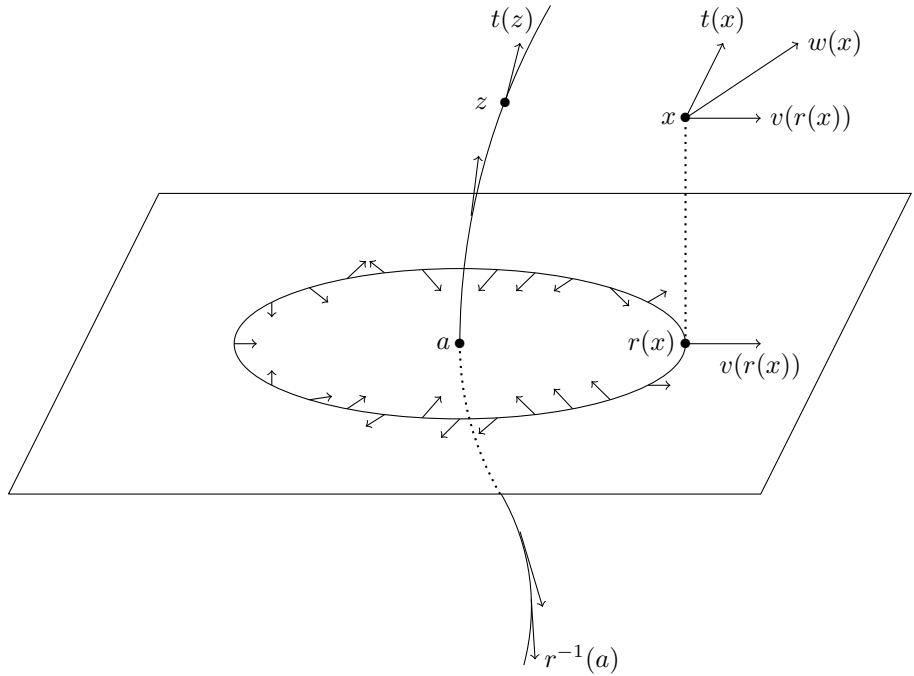
\begin{figure}[h!]
\begin{tikzpicture}
\draw  (-6,-2) -- (-4,2) -- (6,2) -- (4,-2)-- (-6,-2);

\draw (0,0) ellipse (3cm and 1cm);

\node at (0, 0) {$\bullet$};
\node at (0, 0)[left] {$a$};

\node at (3, 0) {$\bullet$};
\node at (3, 0)[left] {$r(x)$};
\draw[->] (3,0) -- (4,0);
\node at (4, 0)[below] {$v(r(x))$};

\node at (3, 3) {$\bullet$};
\node at (3, 3) [left]{$x$};
\draw[->] (3,3) -- (3.5,4);
\node at (3.5, 4)[above] {$t(x)$};
\draw[->] (3,3) -- (4,3);
\node at (4, 3)[right] {$v(r(x))$};
\draw[->] (3,3) -- (4.5,4);
\node at (4.5,4)[right] {$w(x)$};

\draw[dotted, thick] (3,0) -- (3,3);


\draw  (0cm,0cm) arc (0: -30: -9);
\draw[dotted, thick]  (0cm,0cm) arc (0: 30: -4);
\draw  (0.55cm,-2cm) arc (30: -15: 3);


\draw[->](-3,0) -- (-2.7, 0);
\draw[->](-2.5,-0.5527) -- (-2.5,-0.3527);
\draw[->](-2.5, 0.5527) -- (-2.5,0.3527);
\draw[->](2.5,-0.5527) -- (2.8,-0.5527);
\draw[->](2.5,0.5527) -- (2.75,0.7);

\draw[->](-2,-0.745356) -- (-1.7,-0.7);
\draw[->](-2,0.745356) -- (-1.75,0.55);
\draw[->](2,-0.745356) -- (1.75, -0.5);
\draw[->](2,0.745356) -- (2.25, 0.5);

\draw[->](-1.5,-0.866) -- (-1.25, -0.7);
\draw[->](1.5,-0.866) -- (1.25, -0.6);
\draw[->](-1.5, 0.866) -- (-1.25, 1.1);
\draw[->](1.5,0.866) -- (1.25, 0.7);

\draw[->](-1,-0.9428) -- (-1.25, -1.1);
\draw[->](-1, 0.9428) -- (-1.2,  1.1);
\draw[->](1,-0.9428) -- (0.75, -0.7);
\draw[->](1, 0.9428) -- (0.75, 0.7);

\draw[->](-0.5,-0.986) -- (-0.25, -0.7);
\draw[->](-0.5,0.986) -- (-0.25, 0.7);
\draw[->](0.5,-0.986) -- (0.25, -1.2);
\draw[->](0.5,0.986) -- (0.25, 0.7);

\draw[->](0,-1) -- (-0.25, -1.25);


\node at (0.6, 3.2) {$\bullet$};
\draw [->] (0.6, 3.2) -- (0.8, 4);
\node at (0.5, 3.2)[left] {$z$};
\node at (0.7, 4)[above] {$t(z)$};
\draw [->] (0.15, 1.7) -- (0.25, 2.5);

\draw [->] (0.95, -3.4) -- (1, -4.2);
\node at (1, -4.2)[right] {$r^{-1}(a)$};
\draw [->] (0.8, -2.5) -- (1.1, -3.5);

\end{tikzpicture}
\caption{O campo de vetores 
$w(x)$ \'e soma do campo transversal $t(x)$ e do 
campo  paralelo $v(r(x))$.}\label{vizi1}
\end{figure}
\end{center}

Agora, temos um campo de vetores $w$ definido sobre $X_\varepsilon$ satisfazendo as seguintes  propriedades: $w$ tem as mesmas singularidades que $v$ e  os mesmos \'indices nos pontos singulares, e $w$ est\'a saindo de $X_\varepsilon$ ao longo da borda $\partial X_\varepsilon$. 

Para concluir a segunda etapa, vamos usar (e mostrar) o seguinte  resultado  de Hopf: se 
$$w : X_\varepsilon \to \R^3$$
 \'e um campo de vetores suave com singularidades isoladas, saindo de $X_\varepsilon$ ao longo da borda $\partial X_\varepsilon$, ent\~ao a soma $\sum_{a_i} I(w,a_i;X_\varepsilon)$ \'e igual ao grau da 
aplica\c c\~ao de Gauss associada a $n : \partial X_\varepsilon \to \Sp^2$. Isso implica  que a soma 
$\sum_{a_i} I(w,a_i;X_\varepsilon)$ n\~ao depende do campo $w$, e como consequ\^encia, a soma 
$\sum_{a_i}I(v,a_i)$ n\~ao depende do campo $v$ tangente a $\Su$ com singularidades isoladas.

O resultado de Hopf pode ser visto do seguinte jeito: Considaremos uma pequena bola $B_\eta(a_i)$ 
ao redor de cada ponto $a_i$ em $X_\varepsilon$. O raio $\eta<< \varepsilon$ \'e suficientemente pequeno para que
as bolas  estejam no interior de $X_\varepsilon$ e n\~ao se  interseccionam. A variedade 
$$X_\varepsilon \setminus \bigcup_i B_\eta(a_i)$$
\'e uma variedade suave com borda  
$$\partial \left[ X_\varepsilon \setminus \bigcup_i B_\eta(a_i) \right] = 
\partial X_\varepsilon \bigcup - \left[ \bigcup_i \partial B_\eta(a_i) \right],$$
onde sinal ``$-$'' vem do que a orienta\c c\~ao de $\partial B_\eta(a_i) $ como 
borda da bola $B_\eta(a_i) $ \'e a orienta\c c\~ao oposta de $\partial B_\eta(a_i) $ como 
elemento da borda de $X_\varepsilon \setminus \bigcup_i B_\eta(a_i)$. 

Agora, a aplica\c c\~ao  de Gauss
$$
\gamma_w : \partial \left[ X_\varepsilon \setminus \bigcup_i B_\eta(a_i) \right] \to \Sp^2, \quad 
\text{associada a }  \widetilde w(x) = \frac{w(x)} {\Vert w(x) \Vert}, $$
que \'e bem definida sobre a borda de $X_\varepsilon \setminus \bigcup_i B_\eta(a_i) $, 
se estende a uma aplica\c c\~ao 
$$\gamma_w: 
X_\varepsilon \setminus \bigcup_i B_\eta(a_i) \to \Sp^2$$
sem nenhuma singularidade. Com efeito, todas  as singularidades de $\widetilde w$ est\~ao nas bolas 
$B_\eta(a_i)$. 

Sabemos que se um campo de vetores $v$, definido sobre a borda de uma bola $B(a)$, 
se estende sem singularidades dentro de $B(a)$, ent\~ao o \'indice $I(v,a)$ (o seu grau) vale $0$
(veja Observa\c c\~ao \ref{remarque}). 
Isso \'e  um caso particular de uma 
 propriedade importante do grau: se uma aplica\c c\~ao $f : \partial Y \to \Sp^2$ definida sobre a borda de uma variedade $Y$ de dimens\~ao 3 em $\R^3$, se estende sem singularidade no interior de $Y$, ent\~ao 
${\rm grau} (f) =0$. 

Como $\gamma_w$ se estende sem singularidade  a uma aplica\c c\~ao $\gamma_w: 
X_\varepsilon \setminus \bigcup_i B_\eta(a_i) \to \Sp^2$, 
isso implica que o grau de $\gamma_w$ sobre a borda de 
$X_\varepsilon \setminus \bigcup_i B_\eta(a_i)$ vale $0$.
Em outras palavras, 
a soma dos graus de $\gamma_w$ sobre  os componentes da borda de 
$X_\varepsilon \setminus \bigcup_i B_\eta(a_i)$ vale $0$. Isto \'e  
$${\rm grau}(\gamma_w; \partial X_\varepsilon) -  \sum_i {\rm grau}(\gamma_w;\partial B_\eta(a_i)) =0,$$
onde n\~ao esquecemos do sinal ``$-$'' por causa de orienta\c c\~ao. 

Por um lado, $\widetilde w$ est\'a saindo de $X_\varepsilon$ ao longo da sua borda e $\widetilde w$ pode
ser deformado continuamente no campo de vetores $n$  normais, saindo de $X_\varepsilon$ 
ao longo da sua borda. A deforma\c c\~ao (ou homotopia), sendo cont\'inua e com vetores saindo de 
$X_\varepsilon$ ao longo da sua borda,  faz com que o grau, que \'e um n\'umero inteiro, permane\c ca o mesmo, 
o que implica que ${\rm grau}(\gamma_w; \partial X_\varepsilon) = {\rm grau}(n)$. 

Por outro lado, pela defini\c c\~ao do \'indice, ${\rm grau}(\gamma_w;\partial B_\eta(a_i)) = I(w,a_i;X_\varepsilon) = I(v,a_i)$. Finalmente temos 
$$\sum_i I(v,a_i) =  {\rm grau}(n),$$
o que n\~ao depende do campo de vetores $v$ com singularidades isoladas sobre $\Su$. 
Assim, conclu\'imos a prova no caso orient\'avel.

\bigskip

\subsection{Caso n\~ao orient\'avel}

No caso de superf\'icies n\~ao orient\'aveis, considaremos o cobrimento duplo e  orientado. 
Isto \'e, em um ponto $x$ de  uma superf\'icie orientada $\widehat \Su$, consideramos o conjunto de pares $(x,{o}_x)$,  
onde ${o}_x$ \'e uma orienta\c c\~ao local de $\Su$ em $x$. S\~ao duas orienta\c c\~oes locais poss\'iveis, ent\~ao temos uma aplica\c c\~ao $\pi : \widehat \Su \to \Su$, tal que 
$\pi^{-1} (x)$ cont\'em dois pontos (correspondente \`as duas orienta\c c\~oes). 
O cobrimento duplo orientado $\pi : \widehat \Su \to \Su$ \'e um homeomorfismo local.

Por exemplo, no caso do espa\c co projetivo $\PP^2$, o cobrimento duplo orientado \'e a proje\c c\~ao 
can\^onica 
$\pi : \Sp^2 \to \PP^2$. No caso da  garrafa de Klein $K^2$, o cobrimento duplo orientado \'e 
a proje\c c\~ao do toro $\pi : \T^2 \to K^2$. No caso de uma superf\'icie orientada, 
a proje\c c\~ao $\pi : \widehat \Su \to \Su$ \'e isomorfa a $\pi : \Su \times \Z/2 \to \Su$, isto \'e 
$\widehat \Su$ n\~ao \'e conexa e \'e a uni\~ao de dois exemplares de $\Su$. 

Agora, seja $\Su$ uma superf\'icie compacta suave e n\~ao orientada. Por um lado, se $v$ \'e um campo cont\'inuo de vetores sobre $\Su$ 
com singularidades isoladas $a_i$ de \'indice $I(v;a_i)$, ent\~ao, podemos definir 
um levantamento $\hat v$ de $v$ que \'e um  campo cont\'inua de vetores sobre  $\widehat \Su$ 
com pontos singulares isoladas  $a_i^j$ com  $j=1,2$, tais que $\pi (a_i^j) = a_i$. 
Como $\pi$ \'e um homeomorfismo local, temos 
$I(v;a_i^j)= I(v;a_i)$ para cada  $j=1,2$. Ent\~ao $\sum_{i,j}I(\hat v;a_i^j) = 2 \sum_i I(v;a_i)$. 

Por outro lado, podemos definir uma triangula\c c\~ao de $\widehat \Su$ tal que os simplexos de 
$\widehat \Su$ s\~ao imagens inversas dos simplexos de $\Su$. A imagem 
inversa de cada um dos simplexos de $\Su$ consiste em dois simplexos de $\widehat \Su$ de 
mesma dimens\~ao. Os n\'umeros de simplexos $n_0, n_1$ e $n_2$ de $\widehat S$ s\~ao
dobro dos n\'umeros correspondentes de $\Su$. Ent\~ao temos 
$\chi(\widehat \Su) = 2 \chi (\Su)$.

Como $\widehat \Su$ \'e uma superf\'icie compacta suave e orientada, sabemos, 
pela prova no caso orient\'avel,  que 
$\chi(\widehat \Su) = \sum_{i,j}I(\hat v;a_i^j)$. 

Assim obtemos o Teorema de Poincar\'e-Hopf para  superf\'icies n\~ao orient\'aveis:
$$\chi (\Su) = 1/2\cdot \chi(\widehat \Su) = 1/2\; \sum_{i,j}I(v;a_i^j) = \sum_i I(v;a_i).$$

\subsection{Conclus\~oes}\hfill

Ao procurarmos por ``Poincar\'e-Hopf Theorem'' em um site da internet, vemos  v\'arias aplica{\c c}\~oes em mec\^anica, 
f\'isica, qu\'imica, ci\^encias exatas,  economica, etc. 

A raz\~ao mais importante \'e que temos:
\begin{theorem}
{\it Uma variedade $\Su$ suave, compacta  e sem borda admite um campo de vetores tangentes sem singularidades se, e somente se, $\chi(\Su)=0$.}
\end{theorem}

\begin{corollary}\label{cheveux}
\'E imposs\'ivel construir um campo de vetores tangentes \`a esfera $\Sp^2$, sem singularidades. 
O mesmo resultado vale para o plano projetivo $\PP^2$. 
\end{corollary}

A consequ\^encia deste resultado \'e que o fibrado tangente \`a esfera $\Sp^2$ n\~ao \'e trivial. 
O mesmo reultado vale para todas esferas de dimens\~ao par.

\end{document}